\newif\ifpictures
\picturestrue

\documentclass[12pt]{amsart}
\usepackage{amssymb,amsmath}
 \usepackage{amsopn}
 \usepackage{mathabx}
 \usepackage{mathtools}
 \usepackage{xspace}
  \usepackage{hyperref}
 \usepackage[dvips]{graphicx}
\usepackage[arrow,matrix,curve]{xy}
\usepackage {color, tikz}
\usepackage{wasysym}

\numberwithin{equation}{section}
\newtheorem{thm}{Theorem}
\newtheorem{prop}[thm]{Proposition}
\newtheorem{lemma}[thm]{Lemma}
\newtheorem{cor}[thm]{Corollary}

\newtheorem{fact}[thm]{Fact}

\theoremstyle{definition}

\newtheorem{definition}[thm]{Definition}
\newtheorem{example}[thm]{Example}

\numberwithin{thm}{section}

\newcounter{FNC}[page]
\def\newfootnote#1{{\addtocounter{FNC}{2}$^\fnsymbol{FNC}$%
     \let\thefootnote\relax\footnotetext{$^\fnsymbol{FNC}$#1}}}

\newcommand{\N}{\mathbb{N}}

\newcommand{\R}{\mathbb{R}}

\newcommand{\Z}{\mathbb{Z}}

\newcommand{\lf}{\left}
\newcommand{\ri}{\right}

\newcommand{\Lera}{\Leftrightarrow}

\newcommand{\eps}{\varepsilon}

\newcommand{\alp}{\alpha}

\definecolor{DarkGreen}{rgb}{0,0.65,0}

\newcommand{\assum}{{\color{red} \sf $(\clubsuit)$}}

\definecolor{NiceBlue}{rgb}{0.2,0.2,0.75}
\newcommand{\struc}[1]{{\color{NiceBlue} #1}}
\newcommand{\alert}[1]{{\color{red} #1}}

\DeclareMathOperator{\conv}{conv}

\DeclareMathOperator{\New}{New}

\DeclareMathOperator{\sign}{sign}

\def\endexa{\hfill$\hexagon$}

\DeclarePairedDelimiter{\ceil}{\lceil}{\rceil}

\title[An Approach to Constrained Polynomial Optimization via SONC and GP]{An Approach to Constrained Polynomial Optimization via Nonnegative Circuit Polynomials and Geometric Programming}

\author{Mareike Dressler} \author{Sadik Iliman} \author{Timo de Wolff}

\address{Mareike Dressler, Goethe-Universit\"at, FB 12 -- Institut f\"ur Mathematik,
Postfach 11 19 32, 60054 Frankfurt am Main, Germany\medskip}
\email{dressler@math.uni-frankfurt.de}

\address{Sadik Iliman
Frankfurt am Main, Germany\medskip}
\email{sadik.iliman@gmx.net}

\address{Timo de Wolff, Technische Universit\"at Berlin, Institut f\"ur Mathematik, Stra\ss{}e des 17. Juni 136, 10623 Berlin,
 Germany\medskip}

\email{dewolff@math.tu-berlin.de}


\begin{document}

\begin{abstract}
In this article we combine two developments in polynomial optimization. On the one hand, we consider nonnegativity certificates based on sums of nonnegative circuit polynomials, which were recently introduced by the second and the third author. On the other hand, we investigate geometric programming methods for constrained polynomial optimization problems, which were recently developed by Ghasemi and Marshall. We show that the combination of both results yields a new method to solve certain classes of constrained polynomial optimization problems. We test the new method experimentally and compare it to semidefinite programming in various examples.
\end{abstract}

\maketitle

\section{Introduction}

Solving polynomial optimization problems is a key challenge in countless applications like dynamical systems, robotics, control theory, computer vision, signal processing, and economics; e.g. \cite{Blekherman:Parrilo:Thomas,Lasserre:NonnegativeApplications}. It is well-known that polynomial optimization problems are NP-hard in general both in the constrained and in the unconstrained case \cite{Dickinson:Gijben}. Starting with the seminal work of Lasserre in \cite{Lasserre:FirstLasserreRelaxation}, relaxation methods were developed which are significantly faster and provide lower bounds. These methods were studied intensively by means of aspects like exactness and quality of the relaxations \cite{Laurent:deKlerk,Nie:CertifyingLasserre,Nie:Jacobian,Nie:OptimalityConditions}, the speed of the computations \cite{Lasserre:NonnegativeApplications,Parrilo:Sturmfels:MinimizingPolynomialFunctions}, and geometrical aspects of the underlying structures \cite{Blekherman:Volume,Blekherman:PSDandSOS}. A great majority of these results are based on the original approach by Lasserre, called \struc{\emph{Lasserre relaxation}}, which relies on \struc{\textit{semidefinite programming (SDP)}} methods and \struc{\textit{sums of squares (SOS)}} certificates to provide lower bounds for polynomial optimization problems. SDPs can be solved in polynomial time in problem size (up to an $\eps$-error); e.g. \cite[p. 41]{Blekherman:Parrilo:Thomas} and references therein. However, the size of such programs grows exponentially with the number of variables $n$ or the degree $d$ of the polynomials, as its size is given by the number of monomials of $n$-variate monomials of degree at most $d$.

Recently, Ghasemi and Marshall suggested a promising alternative approach both for constrained and unconstrained optimization problems based on \struc{\emph{geometric programming (GP)}} \cite{Ghasemi:Marshall:GP, Ghasemi:Marshall:GP:Semialgebraic}.
GPs can also be solved in polynomial time (up to an $\eps$-error) \cite{nesterov}; see also \cite[Page 118]{Boyd:Kim:Vandenberghe:TutorialOnGP}, but, by experimental results, e.g. \cite{Boyd:Kim:Vandenberghe:TutorialOnGP,Ghasemi:Marshall:GP, Ghasemi:Marshall:GP:Semialgebraic,Ghasemi:Marshall:Lasserre:GP}, in practice the corresponding geometric programs can be solved \textit{significantly} faster than their counterparts in semidefinite programming. 
The lower bounds obtained by Ghasemi and Marshall are, however, by construction worse than lower bounds obtained via semidefinite programming, and they can only be applied in very special cases.

Independent of Ghasemi and Marshall, the second and the third author recently developed a new certificate for nonnegativity of real polynomials called \struc{\textit{sums of nonnegative circuit polynomials (SONC)}} \cite{Iliman:deWolff:Circuits}. SONC certificates are independent of SOS certificates. In \cite{Iliman:deWolff:GP} the second and third author showed that the GP based approach for unconstrained optimization by Ghasemi and Marshall can be generalized crucially via SONC certificates. In consequence, the presented geometric programs are linked to sums of nonnegative circuit polynomials similarly as semidefinite programming relaxations are linked to sums of squares.
Particularly, there exist various classes of polynomials for which the GP/SONC based approach is not only \textit{faster} but, it also yields \textit{better} bounds than the SDP/SOS approach. The reason is that all certificates used by Ghasemi and Marshall are always SOS, while SONCs are not SOS in general; see \cite[Proposition 7.2]{Iliman:deWolff:Circuits}.\\

The first contribution of this article is an extension of the results in \cite{Iliman:deWolff:GP} to constrained polynomial optimization problems. We focus on the class of \struc{\textit{ST-polynomials}}, that are polynomials which have a Newton polytope that is a simplex and which are satisfying some further conditions; see Section \ref{SubSec:PrelimSONC}. The starting point is a general optimization problem from \cite[Section 5]{Iliman:deWolff:GP}, see \eqref{Equ:Constrained}, which provides a lower bound for the constrained problem but which is not a geometric program. Using results from \cite{Ghasemi:Marshall:GP:Semialgebraic}, we relax the program \eqref{Equ:Constrained} into a geometric optimization problem; see program \eqref{Equ:ConstrainedGP} and Theorem \ref{Thm:ConstrainedGP}.
Additionally, we show in Theorem \ref{Thm:ConstrainedSP} that \eqref{Equ:Constrained} can always at least be transformed into a signomial program; see Section \ref{SubSec:PrelimGP} for background information. 
Furthermore, we prove that the new, relaxed geometric program  \eqref{Equ:ConstrainedGP} provides bounds as good as the initial program \eqref{Equ:Constrained} for certain special cases, see Theorem \ref{Thm:ConstrainedGPequality}.

In Section \ref{Sec:ExamplesConstrained}, we provide examples testing our new program \eqref{Equ:ConstrainedGP} in practice and comparing it with semidefinite programming. Moreover, we demonstrate that increasing the degree of a given problem has almost \textit{no} effect on the runtime of our program \eqref{Equ:ConstrainedGP}. This is in sharp contrast to SDPs, where one can nullify increased runtimes induced by high degrees only by additional pre-processing methods, e.g. by exploiting sparsity.

Furthermore, a bound obtained by Ghasemi and Marshall in \cite{Ghasemi:Marshall:GP:Semialgebraic} can never be better than the bound given by the $d$-th Lasserre relaxation for some specific $d$ determined by the degrees of the involved polynomials. In Section \ref{Sec:ExamplesConstrained} we provide examples showing that our program \eqref{Equ:ConstrainedGP} can provide bounds which are better than the particularly $d$-th Lasserre relaxations.\\

The second contribution of this article is to apply polynomial optimization methods based on SONCs and GPs efficiently \textit{beyond} the class of ST-polynomials. In Section \ref{Sec:NonSTPolynomials}, we develop an initial approach based on triangulations of support sets of the involved polynomials. It yields bounds for nonnegativity based on SONC/GP for \textit{arbitrary} polynomials both in the constrained and in the unconstrained case. We provide several examples and compare the new bounds to the ones obtained by SDP based methods.\\

\subsection*{Acknowledgements}

We thank Henning Seidler for various comments and his support on some of the computations. We also thank the anonymous referees, who helped us to improve the article significantly.

The third author was partially supported by the DFG grant WO 2206/1-1.

\section{Preliminaries}
\label{Sec:Preliminaries}

In this section we recall key results about sums of nonnegative circuit polynomials (SONCs) and geometric programming (GP), which are used in this article. 

\subsection{The Cone of Sums of Nonnegative Circuit Polynomials}
\label{SubSec:PrelimSONC}

We denote by \struc{$\R[\mathbf{x}] = \R[x_1,\ldots,x_n]$} the vector space of real $n$-variate polynomials.
Let \struc{$\delta_{ij}$} be the $ij$-Kronecker symbol, let $\struc{\mathbf{e}_i}=(\delta_{i1},\ldots,\delta_{in})$ be the $i$-th standard vector, and let $\struc{A} \subset \N^n$ be a finite set. We denote by \struc{$\conv(A)$}  the convex hull of $A$ and by \struc{$V(A)$} the vertices of $\conv(A)$. We consider polynomials $f \in \R[\mathbf{x}]$ supported on $A$. That is, $f$ is of the form $\struc{f(\mathbf{x})} = \sum_{\boldsymbol{\alp} \in A}^{} f_{\boldsymbol{\alp}}\mathbf{x}^{\boldsymbol{\alp}}$ with $\struc{f_{\boldsymbol{\alp}}} \in \R$, $\struc{\mathbf{x}^{\boldsymbol{\alp}}} = x_1^{\alp_1} \cdots x_n^{\alp_n}$.  We call a lattice point \struc{\textit{even}} if it is in $(2\N)^n$. Furthermore, we denote the Newton polytope of $f$ as $\struc{\New(f)} = \conv\{\boldsymbol{\alp} \in \N^n : f_{\boldsymbol{\alp}} \neq 0\}$.

For a given $A \subset \N^n$ we define $\struc{\Delta(A)} = A \setminus V(A)$. Let $f$ be as before. We denote by $\Delta(f)$ the elements of $\Delta(A)$ which appear as exponents of non-zero terms, that are no monomial squares. I.e., we have
\begin{eqnarray*}
 \struc{\Delta(f)} & = & \{\boldsymbol{\alp} \in \Delta(A) \ : \ f_{\boldsymbol{\alp}} \neq 0 \text{ and } (f_{\boldsymbol{\alp}} < 0 \text{ or }  \boldsymbol{\alp} \notin (2\N)^n)\}.
\end{eqnarray*}

\medskip

A polynomial is nonnegative on the entire $\R^n$ only if the following necessary conditions are satisfied; see e.g. \cite{Reznick:ExtremalPSD}.

\begin{prop}
	Let $A \subset \N^n$ be a finite set and $f \in \R[\mathbf{x}]$ be supported on $A$ such that $\New(f) = \conv(A)$. Then $f$ is nonnegative on $\R^n$ only if:
	\begin{enumerate}
		\item All elements of $V(A)$ are even.
		\item If $\boldsymbol{\alp} \in V(A)$, then the corresponding coefficient $f_{\boldsymbol{\alp}}$ is strictly positive.
	\end{enumerate}
	In other words, if $\boldsymbol{\alp} \in V(A)$, then the term $f_{\boldsymbol{\alp}} \mathbf{x}^{\boldsymbol{\alp}}$ has to be a \struc{\emph{monomial square}}.
	\label{Prop:NecessaryConditions}
\end{prop}

The statement remains true for real Laurent polynomials $g \in \R[\mathbf{x}^{\pm 1}] = \R[x_1^{\pm 1},\ldots,x_n^{\pm 1}]$, since we can consider $g$ as a polynomial $f$ divided by a monomial square $\mathbf{x}^{\boldsymbol{\alp}}$ for an even $\boldsymbol{\alp}$; this is of relevance in Section \ref{Sec:NonSTPolynomials}. For the remainder of the article, we assume that these necessary conditions in Proposition \ref{Prop:NecessaryConditions} are satisfied including $\New(f) = \conv(A)$. For simplicity, we denote this assumption by the symbol \assum{}  from now on.\\

In what follows we consider the class of \textit{ST-polynomials}. For further details about the following objects defined in this section see \cite{deWolff:Circuits:OWR,Iliman:deWolff:Circuits,Iliman:deWolff:GP}; see also \cite{Fidalgo:Kovacec,Ghasemi:Marshall:GP,Ghasemi:Marshall:GP:Semialgebraic}.
\begin{definition}
Let $f \in \R[\mathbf{x}]$ be supported on $A \subset \N^n$ such that \assum{} holds. Then $f$ is called an \struc{\emph{ST-polynomial}} if it is of the form
\begin{eqnarray}
 \struc{f(\mathbf{x})} & = & \sum_{j=0}^r f_{\boldsymbol{\alp}(j)} \mathbf{x}^{\boldsymbol{\alp}(j)} +\sum_{\boldsymbol{\beta} \in \Delta(A)} f_{\boldsymbol{\beta}} \mathbf{x}^{\boldsymbol{\beta}}, \label{Equ:STPolynomial}
\end{eqnarray}
with $\struc{r} \leq n$, exponents $\struc{\boldsymbol{\alp}(j)}$ and $\struc{\boldsymbol{\beta}}$, and coefficients $\struc{f_{\boldsymbol{\alp}(j)}}, \struc{f_{\boldsymbol{\beta}}}$, for which the following conditions hold:

\begin{description}
 \item[(ST1)] The points $\boldsymbol{\alp}(0), \boldsymbol{\alp}(1),\ldots,\boldsymbol{\alp}(r)$ are affinely independent and equal $V(A)$.
 \item[(ST2)] Every exponent $\boldsymbol{\beta} \in \Delta(A)$ can be written uniquely as 
 \begin{eqnarray*}
  & & \boldsymbol{\beta} \ = \ \sum_{j=0}^r \lambda_j^{(\boldsymbol{\beta})} \boldsymbol{\alp}(j) \ \text{ with } \ \lambda_j^{(\boldsymbol{\beta})} \ \geq \ 0 \ \text{ and } \  \sum_{j=0}^r \lambda_j^{(\boldsymbol{\beta})} \ = \ 1.
 \end{eqnarray*}
\end{description}
\label{Def:STPolynomial}
\endexa
\end{definition}

Note that hypotheses (ST1) and (ST2) imply that  $V(A) =\{\boldsymbol{\alp}(0),\ldots,\boldsymbol{\alp}(r)\}$ is the vertex set of an $r$-dimensional simplex. By the assumption \assum{} it consists of even lattice points, and it coincides with $\New(f)=\conv(A)$.  
The $\struc{\lambda_j^{(\boldsymbol{\beta})}}$ denote the barycentric coordinates of $\boldsymbol{\beta}$ relative to the vertices $\boldsymbol{\alp}(j)$ with $j=0,\ldots,r$. The ``ST'' in ``ST-polynomial'' is short for ``simplex tail''. The tail part is
given by the sum $\sum_{\boldsymbol{\beta}\in \Delta(A)} f_{\boldsymbol{\beta}} \mathbf{x}^{\boldsymbol{\beta}},$ while the other terms define the simplex part. If an ST-polynomial $f$ has a tail part consisting of at most one term, then we call $f$ a \struc{\emph{circuit polynomial}}.

\medskip

Nonnegativity of ST-polynomials is closely related to an invariant called the \textit{circuit number}.

\begin{definition}
Let $f$ be an ST-polynomial with support set $A$. For every $\boldsymbol{\beta} \in \Delta(A)$ we define the corresponding \struc{\textit{circuit number}} as
\begin{eqnarray}
 \struc{\Theta_f(\boldsymbol{\beta})} \ = \ \prod_{j\in {\rm nz}(\boldsymbol{\beta})} \left(\frac{f_{\boldsymbol{\alp}(j)}}{\lambda_j^{(\boldsymbol{\beta})}}\right)^{\lambda_j^{(\boldsymbol{\beta})}} \label{Equ:DefCircuitNumber}
\end{eqnarray}
with $\struc{{\rm nz}(\boldsymbol{\beta})}=\{j\in \{0,\ldots,r\}:\lambda_j^{(\boldsymbol{\beta})}\neq 0\}$, $f_{\boldsymbol{\alp}(j)}$, and $\lambda_j^{(\boldsymbol{\beta})}$ as before.
\endexa
\end{definition}

The terms ``circuit polynomial'' and ``circuit number'' are chosen since $\boldsymbol{\beta}$ and the $\boldsymbol{\alp}(j)$ with $j \in {\rm nz}(\boldsymbol{\beta})$ form a \struc{\textit{circuit}}; this is a minimally affine dependent set, see e.g. \cite{Oxley:MatroidTheory}.

A fundamental fact is that nonnegativity of a circuit polynomial $f$ can be decided by comparing its tail coefficient
$f_{\boldsymbol{\beta}}$ with its corresponding circuit number $\Theta_f(\boldsymbol{\beta})$ alone.

\begin{thm}[\cite{Iliman:deWolff:Circuits}, Theorem 3.8]
\label{Thm:Positiv}
 Let $f$  be a  circuit polynomial with unique tail term $f_{\boldsymbol{\beta}} \mathbf{x}^{\boldsymbol{\beta}}$ and let $\Theta_f(\boldsymbol{\beta})$ be the corresponding circuit number, as defined in \eqref{Equ:DefCircuitNumber}. 
Then the following statements are equivalent:
\begin{enumerate}
 \item $f$ is nonnegative.
 \item $|f_{\boldsymbol{\beta}}| \leq \Theta_f(\boldsymbol{\beta})$ and $\boldsymbol{\beta} \not \in (2\N)^n$ \quad or \quad $f_{\boldsymbol{\beta}}\geq -\Theta_f(\boldsymbol{\beta})$ and $\boldsymbol{\beta} \in (2\N)^n$.
\end{enumerate}
\end{thm}

Note that (2) can be equivalently stated as: $|f_{\boldsymbol{\beta}}| \leq \Theta_f(\boldsymbol{\beta})$ or $f$ is a sum of monomial squares. 

Writing a polynomial as a sum of nonnegative circuit polynomials is a certificate of nonnegativity. We denote
by  \struc{SONC} the class of polynomials that are \struc{\it sums of nonnegative circuit polynomials} or the property of a polynomial to be in this class. 

\subsection{Geometric Programming}
\label{SubSec:PrelimGP}

Geometric programming was introduced in \cite{Duffin:Peterson:Zener:Book}. It is a convex optimization problem and has applications for example in nonlinear network flow problems, optimal control, optimal location problems, chemical equilibrium problems and particularly in circuit design problems.

\begin{definition}
A function $p : \mathbb R_{>0}^n\to \mathbb R$ of the form $\struc{p(\mathbf{z})} = p(z_1,\ldots,z_n) = cz_1^{\alp_1}\cdots z_n^{\alp_n}$ with $c > 0$ and $\alp_i \in \mathbb R$ is called a \struc{\emph{monomial (function)}}. A sum $\struc{\sum_{i=0}^k c_iz_1^{\alp_{1}(i)}\cdots z_n^{\alp_{n}(i)}}$ of monomials with $c_i > 0$ is called a \struc{\emph{posynomial (function)}}.

A \struc{\emph{geometric program} (GP)} has the following form:
\begin{eqnarray}
\label{Equ:GP}
\begin{cases} 
\text{minimize} & p_0(\mathbf{z}), \\ 
\text{subject to:} &  
\begin{array}{cl}
 (1) & p_i(\mathbf{z}) \leq 1 \ \text{ for all } \ 1\leq i\leq m, \\
 (2) & q_j(\mathbf{z}) = 1 \ \text{ for all } \ 1\leq j\leq l, \\
\end{array}
\end{cases}
\end{eqnarray}
where $p_0,\dots,p_m$ are posynomials and $q_1,\dots,q_l$ are monomial functions.
\label{Def:GP}
\endexa
\end{definition}

Geometric programs can be solved with interior point methods. In \cite{nesterov}, the authors prove  worst-case polynomial time complexity of this method; see also \cite[Page 118]{Boyd:Kim:Vandenberghe:TutorialOnGP}. A \struc{\textit{signomial program}} is given like a geometric program except that the coefficients $c_i$ of the involved posynomials can be arbitrary real numbers.

For an introduction to geometric programming, signomial programming, and an overview about applications see \cite{Boyd:Kim:Vandenberghe:TutorialOnGP, Boyd:CO}.

\subsection{SONC Certificates via Geometric Programming in the Unconstrained Case}
\label{SubSec:PrelimGPCertificates}

In this section we recall the main results from \cite{Iliman:deWolff:GP} about SONC certificates obtained via geometric programming for unconstrained polynomial optimization problems. These results always require that the polynomial in the optimization problem is an ST-polynomial in the sense of Section \ref{SubSec:PrelimSONC}.

\begin{thm}(\cite[Theorems 3.4 and 3.5]{Iliman:deWolff:GP})
\label{Thm:GPUnconstrainedCase}
 Assume that $f$ is an ST-polynomial as in \eqref{Equ:STPolynomial} with $\boldsymbol{\alp}(0)=0$ and let $k\in \R.$   Suppose
 that for every $(\boldsymbol{\beta},j) \in \Delta(f) \times \{1,\ldots,r\}$ 
there exists an $a_{\boldsymbol{\beta},j} \geq 0,$ such that:
\begin{enumerate}
 \item $a_{\boldsymbol{\beta},j} > 0$ if and only if $\lambda_j^{(\boldsymbol{\beta})}>0$,
 \item $|f_{\boldsymbol{\beta}}| \leq \prod\limits_{j\in {\rm nz}(\boldsymbol{\beta})} \left(\frac{a_{\boldsymbol{\beta},j}}{\lambda_j^{(\boldsymbol{\beta})}}\right)^{\lambda_j^{(\boldsymbol{\beta})}}$ 
     \, for every $\boldsymbol{\beta} \in \Delta(f)$ with $\lambda_{0}^{(\boldsymbol{\beta})}=0$,
 \item $f_{\boldsymbol{\alp}(j)} \geq \sum\limits_{\boldsymbol{\beta}\in \Delta(f)} a_{\boldsymbol{\beta},j}$ for all $1 \leq j \leq r,$ 
 \item $f_{\boldsymbol{\alp}(0)}-k  \geq \sum\limits_{\substack{\boldsymbol{\beta} \in \Delta(f) \\ \lambda_0^{(\boldsymbol{\beta})} \neq 0}}  \lambda_0^{(\boldsymbol{\beta})} |f_{\boldsymbol{\beta}}|^{1/\lambda_0^{(\boldsymbol{\beta})}} 
     \prod\limits_{\substack{j\in {\rm nz}(\boldsymbol{\beta}) \\ j\geq 1}}  
         \left(\frac{\lambda_j^{(\boldsymbol{\beta})}}{a_{\boldsymbol{\beta},j}}\right)^{\lambda_j^{(\boldsymbol{\beta})}/\lambda_0^{(\boldsymbol{\beta})}}.$
\end{enumerate}
Then $f-k$ is a sum of nonnegative circuit polynomials $g_1,\ldots,g_s$ such that $s := |\Delta(f)|$, and for every $g_i$ the Newton polytope $\New(g_i)$ is a face of $\New(f)$.

 Let $\struc{f_{\rm sonc}}$ be the supremum of all $k \in \R$ such that for every $\boldsymbol{\beta} \in \Delta(f)$ there exist nonnegative reals 
        $a_{\boldsymbol{\beta},1},\ldots, a_{\boldsymbol{\beta},r}$ such that the conditions $(1)$ to $(4)$ are satisfied. Then $f_{\rm sonc}$ coincides with the supremum of all $k\in \R$ such that there exist nonnegative circuit polynomials $g_1, g_2, \ldots, g_s$
whose Newton polytopes are faces of $\New(f)$ and which satisfy $f-k =\sum_{i=1}^s g_i.$ 
\label{Thm:STPolynomialNN}
 \end{thm}

We remark that Theorem \ref{Thm:STPolynomialNN} remains valid if all terms of an ST-polynomial are multiplied by a monomial square $\mathbf{x}^{\boldsymbol{\alp}(0)} \in (2\N)^n$. In this case, condition needs to be replaced by
\begin{enumerate}
	\item[$(4')$] $(f_{\boldsymbol{\alp}(0)}-k)\mathbf{x}^{\boldsymbol{\alp}(0)}  \geq \sum\limits_{\substack{\boldsymbol{\beta} \in \Delta(f) \\ \lambda_0^{(\boldsymbol{\beta})} \neq 0}}  \lambda_0^{(\boldsymbol{\beta})} |f_{\boldsymbol{\beta}}|^{1/\lambda_0^{(\boldsymbol{\beta})}} 
	\prod\limits_{\substack{j\in {\rm nz}(\boldsymbol{\beta}) \\ j\geq 1}}  
	\left(\frac{\lambda_j^{(\boldsymbol{\beta})}}{a_{\boldsymbol{\beta},j}}\right)^{\lambda_j^{(\boldsymbol{\beta})}/\lambda_0^{(\boldsymbol{\beta})}}.$
\end{enumerate} 
We then obtain a nonnegativity bound for the (positive) the coefficient of the term $\mathbf{x}^{\boldsymbol{\alp}(0)}$ instead of the constant term. We apply this slightly more general version of Theorem \ref{Thm:STPolynomialNN} in Section \ref{Sec:NonSTPolynomials}.
 
For the special case of scaled standard simplices Theorem \ref{Thm:GPUnconstrainedCase} was shown earlier by Ghasemi and Marshall \cite[Theorem 3.1]{Ghasemi:Marshall:GP}. In this special case every sum of nonnegative circuit polynomials is also a sum of binomial squares which is not true in general. For example, the Motzkin polynomial is an ST-polynomial with one interior term, which is not even a SOS.

\smallskip

Theorem \ref{Thm:STPolynomialNN} states
\begin{eqnarray*}
 f_{\rm sonc} & = & \sup\{k \in \R \ : \ f - k \mathbf{x}^{\boldsymbol{\alp}(0)} \text{ is a SONC} \,\}.
\end{eqnarray*}
 
The bound $f_{\rm sonc}$ is given by a geometric program \cite[Corollary 4.2]{Iliman:deWolff:GP}:

\begin{cor}\label{Cor:GP}
Let $f\in \R[\mathbf{x}]$ be an ST-polynomial. Let $R$ be the subset of an $r |\Delta(f)|$-dimensional real space given by
\begin{eqnarray*}
 \struc{R} & = & \{(a_{\boldsymbol{\beta},j}) \ : \ a_{\boldsymbol{\beta},j} \in \R_{> 0} \text{ for every } \boldsymbol{\beta}\in \Delta(f) \text{ and } j \in {\rm nz}(\boldsymbol{\beta})\}.
\end{eqnarray*}
Then $f_{\rm sonc}=f_{\boldsymbol{\alp}(0)}-m^*,$ where $\struc{m^*}$ is given as the output of the following geometric program:
\begin{eqnarray*}
\begin{cases}
\emph{minimize} & \sum\limits_{\substack{\boldsymbol{\beta} \in \Delta(f) \\ \lambda_0^{(\boldsymbol{\beta})} \neq 0}}\lambda_0^{(\boldsymbol{\beta})} |f_{\boldsymbol{\beta}}|^{1/\lambda_0^{(\boldsymbol{\beta})}} 
    \prod\limits_{\substack{ j\in {\rm nz}(\boldsymbol{\beta})\\ j\geq 1}}
  \left(\frac{\lambda_j^{(\boldsymbol{\beta})}}{a_{\boldsymbol{\beta},j}}\right)^{\lambda_j^{(\boldsymbol{\beta})}/\lambda_0^{(\boldsymbol{\beta})}}
\emph{ over the subset } R' \emph{ of } R \\
& \\
\emph{defined by:} &
\begin{array}{cl}
 (1) & \sum\limits_{\boldsymbol{\beta}\in \Delta(f)} (a_{\boldsymbol{\beta},j}/f_{\boldsymbol{\alp}(j)} ) \leq 1 \emph{ for every } 1 \leq j \leq r, \\
 (2) & |f_{\boldsymbol{\beta}}| \prod\limits_{j\in {\rm nz(\boldsymbol{\beta})}} 
                    \left(\frac{\lambda_j^{(\boldsymbol{\beta})}}{a_{\boldsymbol{\beta},j}}\right)^{ \lambda_j^{(\boldsymbol{\beta})}} \leq 1 \emph{ for every } \boldsymbol{\beta}\in \Delta(f) \emph{ with } \lambda_0^{(\boldsymbol{\beta})}=0. \\
\end{array}
\end{cases}
\end{eqnarray*}
\end{cor}

Hence, the optimal bound to find a SONC decomposition of an ST-polynomial is provided by geometric programming. Since a polynomial with a SONC decomposition is nonnegative, geometric programming can be used to find certificates of nonnegativity.\\

Following the literature, e.g. \cite{Blekherman:Parrilo:Thomas,Laurent:Survey}, we define a \struc{\textit{global polynomial optimization problem}} for some $f \in \R[\mathbf{x}]$ as the problem to determine the real number
\begin{eqnarray*}
\struc{f^*} & = &\inf\{f(\mathbf{x}) \ : \ \mathbf{x}\in\R^n\} \ = \ \sup\{\lambda\in\R \ : \ f - \lambda \geq 0\}.
\end{eqnarray*}
 One can find a lower bound for $f^*$ by relaxing the nonnegativity condition in the above problem to finding the real number
\begin{eqnarray*}
	\struc{f_{\rm sos}} & = & \sup\left\{\lambda \in \mathbb R \ : \ f - \lambda = \sum_{i=1}^k q_i^2 \ \text{ for some } \ q_i\in\R[\mathbf{x}]\right\}.
\end{eqnarray*}
 The bound $f_{\rm sos}$ for the optimal SOS decomposition of $f$ can be determined by semidefinite programming. By construction, we have $f_{\rm sos} \leq f^*$; see \cite{Lasserre:NonnegativeApplications}.

A key observation is that the bounds obtained by this approach can be better than the ones obtained by SDP as the following result shows; see \cite[Corollary 3.6]{Iliman:deWolff:GP}.

\begin{cor}\label{cor:gpbessersonc}
Let $f$ be an ST-polynomial with $\Delta(A) = \Delta(f)$ such that $\Delta(f)$ is contained in the interior of $\New(f)$. Let $\boldsymbol{\alp}(0)$ be the origin and suppose that there exists a vector $\mathbf{v}\in (\R^*)^n$ such that $f_{\boldsymbol{\alp}} \cdot \mathbf{v}^{\boldsymbol{\alp}} <0$ for all $\boldsymbol{\alp}\in \Delta(f)$. Then
\begin{eqnarray*}
 f_{\rm sonc} \ = \ f^* \ \geq \ f_{\rm sos}.
\end{eqnarray*}
\end{cor}

\subsection{SONC Certificates for the Constrained Case}
\label{SubSec:PrelimConstraint}

In this subsection we restate facts from \cite[Section 5]{Iliman:deWolff:GP} about SONC certificates applied to constrained polynomial optimization problems.

Let $f, g_1, \ldots,g_s$ be elements of the polynomial ring $\R[\mathbf{x}]$ and let 
\begin{eqnarray*}
\struc{K} \ = \ \{\mathbf{x} \in \mathbb R^n \ : \ g_i(\mathbf{x}) \geq 0 \text{ for all } i = 1,\ldots,s\}
\end{eqnarray*}
be a basic closed semialgebraic set defined by $g_1,\ldots,g_s$. We consider the constrained polynomial optimization problem 
\begin{eqnarray*}
\struc{f_K^*} \ = \ \inf_{\mathbf{x} \in K}^{}f(\mathbf{x}) \ = \ \sup\{\gamma \in \R \ : \ f(\mathbf{x}) - \gamma \geq 0 \;\text{ for all } \mathbf{x} \in K \}. 
\end{eqnarray*}
If $s=0$, then we have no $g_i$ and therefore $K=\R^n$, which leads to the global optimization problem explained in Section \ref{SubSec:PrelimGPCertificates}.

To obtain a general lower bound for $f$ on $K$ which is computable by geometric programming we 
replace the considered polynomials by a new function. Let
\begin{eqnarray}
\struc{G(\boldsymbol{\mu})(\mathbf{x})} \ = \ f(\mathbf{x}) - \sum_{i=1}^{s} \mu_ig_i(\mathbf{x}) \ = \ - \sum_{i=0}^{s} \mu_ig_i(\mathbf{x}) \label{Equ:DefG}
\end{eqnarray}

for $\struc{\boldsymbol{\mu}} = (\mu_1,\dots,\mu_s) \in  \R_{\geq 0}^s$, $\struc{g_0} = -f$ and $\struc{\mu_0} = 1$. For every fixed $\boldsymbol{\mu}^* \in  \R_{\geq 0}^s$ the function $\struc{G(\mathbf{x})} = G(\boldsymbol{\mu}^*)(\mathbf{x})$ is a polynomial in $\R[\mathbf{x}]$. Following an argument in \cite{Ghasemi:Marshall:GP:Semialgebraic} we can assume that all monomial squares of $-g_i$ are vertices of $\New(G(\boldsymbol{\mu}))$: One can reduce to this case by neglecting all monomial squares not corresponding to such a vertex. That is, for all $i=0,\ldots,s$ one can replace $g_i$ by $\tilde{g}_i$, which resemble $g_i$ without monomial squares of $-g_i$ in the interior of $\New(G(\boldsymbol{\mu}))$. Then $-\tilde{g}_i\leq -g_i$ on $\R^n$ for $i=0,\ldots,s$, thus, $K\subseteq \tilde{K}$, where $\tilde{K}=\{\mathbf{x} \in \mathbb R^n \ : \ \tilde{g}_i(\mathbf{x}) \geq 0,\,\, 1\leq i\leq s\}$, as well as $f_{\tilde{K}}^*\leq f_K^*$. 

Let $\struc{A_i} \subset \N^n$ be the support of the polynomial $g_i$ for $i = 0,\ldots,s$ and let $\struc{A} = \bigcup_{i = 0}^s A_i$ be the union of all supports of polynomials $g_i$. We remark that while we consider a fixed support the Newton polytope of $G(\boldsymbol{\mu})$ is not invariant in general since certain $\mu_i$ might equal $0$ or term cancellation might occur. If for some $\boldsymbol{\mu} \in \R_{\geq 0}^s$ the polynomial $G(\boldsymbol{\mu})$ is an ST-polynomial, then we assume that $\New(G(\boldsymbol{\mu})) = \conv(A)$ and $V(A) = \{\boldsymbol{\alp}(0),\ldots,\boldsymbol{\alp}(r)\}  \subset (2\N)^n$ and we denote $\struc{G(\boldsymbol{\mu})_{\rm sonc}}$ as the optimal value of the geometric program in Corollary \ref{Cor:GP}. Theorem \ref{Thm:STPolynomialNN} implies that $G(\boldsymbol{\mu}) - G(\boldsymbol{\mu})_{\rm sonc} \mathbf{x}^{\boldsymbol{\alp}(0)} \geq 0$ and $G(\boldsymbol{\mu})_{\rm sonc} \in \R$ is the maximal possible choice for nonnegativity. Hence, we obtain a bound for the coefficient of the term  $\mathbf{x}^{\boldsymbol{\alp}(0)}$ depending on the other coefficients of $G(\boldsymbol{\mu})$ certifying nonnegativity of $G(\boldsymbol{\mu})$. If $G(\boldsymbol{\mu})$ is not an ST-polynomial for some $\boldsymbol{\mu} \in \R_{\geq 0}^s$, then we set $G(\boldsymbol{\mu})_{\rm sonc} = -\infty$, since the corresponding geometric program is infeasible. Thus, by \eqref{Equ:DefG}, if $\boldsymbol{\mu}$ is fixed, then $G(\boldsymbol{\mu})_{\rm sonc}$ is a lower bound for $f$ on the semialgebraic set $K$ regarding the coefficient of $\mathbf{x}^{\boldsymbol{\alp}(0)}$. Let $\struc{\mathbf{g}} = (g_1,\dots ,g_s)$. We define
\begin{eqnarray*}
 \struc{s(f,\mathbf{g})} \ = \ \sup\{G(\boldsymbol{\mu})_{\rm sonc} : \boldsymbol{\mu} \in \R_{\geq 0}^s\}.
\end{eqnarray*}
Thus, we have particularly for $\boldsymbol{\alp}(0) = \mathbf{0}$:
\begin{eqnarray}
 s(f,\mathbf{g}) \ \leq \ f_K^*. \label{Equ:fKStern}
\end{eqnarray}

For every fixed $\boldsymbol{\mu}$ the bound $G(\boldsymbol{\mu})_{\rm sonc}$ is computable by a geometric program. Unfortunately, this does not imply that the supremum is computable by a GP as well. However, following ideas by Ghasemi and Marshall \cite{Ghasemi:Marshall:GP} the second and third author presented a general optimization program for a lower bound of $s(f,\mathbf{g})$ in \cite{Iliman:deWolff:GP}, which is a geometric program under special conditions. We recall these results in what follows.

We define $\struc{\Delta(A)}$ in the sense of Section \ref{SubSec:PrelimSONC} as the set of exponents of the tail terms of $G(\boldsymbol{\mu})$ and $\struc{\Delta(G(\boldsymbol{\mu}))} \subseteq \Delta(A)$ as the set of exponents which have a non-zero coefficient and are not a monomial square. Moreover, we define $\struc{\Delta(G)} = \bigcup_{\boldsymbol{\mu} \in \R^s_{\geq 0}}\Delta(G(\boldsymbol{\mu}))$. Note that $\Delta(G(\boldsymbol{\mu})) \subseteq \Delta(G) \subseteq \Delta(A)$ for all $\boldsymbol{\mu}$. We have by Section \ref{SubSec:PrelimSONC}, Definition \ref{Def:STPolynomial}
\begin{eqnarray*}
 & & G(\boldsymbol{\mu})(\mathbf{x}) \ = \ - \sum_{i = 0}^s \mu_i g_i(\mathbf{x}) \ = \ \sum_{j = 0}^r G(\boldsymbol{\mu})_{\boldsymbol{\alp}(j)} \mathbf{x}^{\boldsymbol{\alp}(j)} + \sum_{\boldsymbol{\beta} \in \Delta(G)} G(\boldsymbol{\mu})_{\boldsymbol{\beta}} \mathbf{x}^{\boldsymbol{\beta}}
\end{eqnarray*}
with coefficients $\struc{G(\boldsymbol{\mu})_{\boldsymbol{\alp}(j)}},\struc{G(\boldsymbol{\mu})_{\boldsymbol{\beta}}} \in \R$ depending on $\boldsymbol{\mu}$. We set the coefficients $G(\boldsymbol{\mu})_{\boldsymbol{\beta}} = 0$ for all $\boldsymbol{\beta} \in \Delta(G) \setminus \Delta(G(\boldsymbol{\mu}))$.

As before, we denote by $\{\lambda_0^{(\boldsymbol{\beta})},\dots,\lambda_r^{(\boldsymbol{\beta})}\}$ the barycentric coordinates of the lattice point $\boldsymbol{\beta} \in \Delta(A)$ with respect to the vertices of the simplex $\New(G(\boldsymbol{\mu})) = \conv(A)$. We define for every $\boldsymbol{\beta} \in \Delta(G)$ a set
\begin{eqnarray*}
 \struc{R_{\boldsymbol{\beta}}} & = & \{\mathbf{a}_{\boldsymbol{\beta}} \ : \ \mathbf{a}_{\boldsymbol{\beta}} = (a_{\boldsymbol{\beta},1},\ldots,a_{\boldsymbol{\beta},r}) \in \R_{>0}^r\}.
\end{eqnarray*}
Furthermore, we define the nonnegative  real set $R$ as
\begin{eqnarray*}
\struc{R} & = & [0,\infty)^s \times \bigtimes_{\boldsymbol{\beta} \in \Delta(G)} (R_{\boldsymbol{\beta}} \times \R_{\geq 0}).
\end{eqnarray*}
Hence, $R$ is the Cartesian product of $[0,\infty)^s$ and $|\Delta(G)|$ many copies $\R_{>0}^r \times \R_{\geq 0}$; each given by one $R_{\boldsymbol{\beta}}$ with $\boldsymbol{\beta} \in \Delta(G)$ and one $\R_{\geq 0}$. We define the function $p$ from $R$ to $\R_{\geq 0}$ as
\begin{eqnarray*}
 & & \struc{p(\boldsymbol{\mu},\{(\mathbf{a}_{\boldsymbol{\beta}},b_{\boldsymbol{\beta}}) \ : \ \boldsymbol{\beta} \in \Delta(G)\})} \ = \ \\
 & & \sum_{i=1}^{s} \mu_i g_{i,\boldsymbol{\alp}(0)} + 
 \sum_{\substack{\boldsymbol{\beta} \in \Delta(G) \\ \lambda_0^{(\boldsymbol{\beta})} \neq 0}}^{}\lambda_0^{(\boldsymbol{\beta})} \cdot b_{\boldsymbol{\beta}}^{\frac{1}{\lambda_0^{(\boldsymbol{\beta})}}}\cdot \prod_{\substack{j\in {\rm nz}(\boldsymbol{\beta}) \\ j\geq 1}} \left(\frac{\lambda_j^{(\boldsymbol{\beta})}}{a_{\boldsymbol{\beta},j}}\right)^{\frac{\lambda_j^{(\boldsymbol{\beta})}}{\lambda_0^{(\boldsymbol{\beta})}}}
\end{eqnarray*}
where, as before, $\struc{\boldsymbol{\alp}(0)}$ is a vertex of $\New(G(\boldsymbol{\mu}))$ and $\struc{g_{i,\boldsymbol{\alp}(0)}}$ is the coefficient of the monomial $\mathbf{x}^{\boldsymbol{\alp}(0)}$ in the polynomial $g_i$.

For the coefficient $G(\boldsymbol{\mu})_{\boldsymbol{\beta}}$ of the term with exponent $\boldsymbol{\beta}$ of $G(\boldsymbol{\mu})$ we use the notation $G(\boldsymbol{\mu})_{\boldsymbol{\beta}} = - \sum_{i = 0}^s \mu_i \cdot \struc{g_{i,\boldsymbol{\beta}}}$. In other words, $G(\boldsymbol{\mu})_{\boldsymbol{\beta}}$ is a linear form in the $\mu_i$'s given by the coefficients of the polynomials $g_i$; analogously for $G(\boldsymbol{\mu})_{\boldsymbol{\alp}(j)}$. We consider the following optimization problem:

\begin{eqnarray}
& & \begin{cases}
\text{minimize} &
p(\boldsymbol{\mu},\{(\mathbf{a}_{\boldsymbol{\beta}},b_{\boldsymbol{\beta}}) \ : \ \boldsymbol{\beta} \in \Delta(G)\})
\text { over the subset of } R
\\
& \\
\text{defined by:} & 
\begin{array}{cl}
 (1) & \sum\limits_{\boldsymbol{\beta}\in\Delta(G)} a_{\boldsymbol{\beta},j} \, \leq \, G(\boldsymbol{\mu})_{\boldsymbol{\alp}(j)} \ \text{ for all } \ 1 \leq j \leq r,  \\
 (2) & \prod\limits_{j\in {\rm nz}(\boldsymbol{\beta})} \left(\frac{a_{\boldsymbol{\beta},j}}{\lambda_j^{(\boldsymbol{\beta})}}\right)^{\lambda_j^{(\boldsymbol{\beta})}} \mspace{-12mu} \geq \,  b_{\boldsymbol{\beta}} \ \text{ for every } \boldsymbol{\beta}\in \Delta(G) \text{ with } \lambda_0^{(\boldsymbol{\beta})}=0, \text{ and }\\
(3) & |G(\boldsymbol{\mu})_{\boldsymbol{\beta}}| \leq b_{\boldsymbol{\beta}} \ \text{ for every } \boldsymbol{\beta}\in \Delta(G) \text{ with } \lambda_0^{(\boldsymbol{\beta})}\neq0. \\
\end{array}
\end{cases}
\label{Equ:Constrained}
\end{eqnarray}

In \cite[Theorems 5.1 and 5.2]{Iliman:deWolff:GP} the second and third author show the following theorem.

\begin{thm}
Let $\struc{\gamma}$ be the optimal value of the optimization problem \eqref{Equ:Constrained}. Then we have $f_{\boldsymbol{\alp}(0)} - \gamma \leq s(f,\mathbf{g})$. The optimization problem \eqref{Equ:Constrained} restricted to $\boldsymbol{\mu} \in (0,\infty)^s$ is a signomial program if for every $\boldsymbol{\beta} \in \Delta(G)$ it holds that $G(\boldsymbol{\mu})_{\boldsymbol{\beta}}$ has the same sign for every choice of $\boldsymbol{\mu}$. 

Assume additionally that every linear form $G(\boldsymbol{\mu})_{\boldsymbol{\alp}(j)} = - \sum_{i = 0}^s \mu_i \cdot g_{i,\boldsymbol{\alp}(j)}$ corresponding to a vertex $\boldsymbol{\alp}(j)$ of $\New(G(\boldsymbol{\mu}))$ has only one summand and is strictly positive. Assume moreover that for all $\boldsymbol{\beta} \in \Delta(G)$ the linear form $G(\boldsymbol{\mu})_{\boldsymbol{\beta}} = - \sum_{i = 0}^s \mu_i \cdot g_{i,\boldsymbol{\beta}}$ has only positive terms. If furthermore all $g_{i,\boldsymbol{\alp}(0)}$ for $1 \leq i \leq s$ are greater than or equal to zero, then \eqref{Equ:Constrained} is a geometric program.
\label{Thm:GPOldConstrainedCase}
\end{thm}

\section{Constrained Polynomial Optimization via Signomial and Geometric Programming}
\label{Sec:RelaxToGP}

In this section, we provide relaxations of the program \eqref{Equ:Constrained} following ideas of Ghasemi and Marshall in \cite{Ghasemi:Marshall:GP:Semialgebraic}. The goal is to weaken the assumptions which are needed to obtain a geometric program or at least a signomial program. We provide such relaxations in the programs \eqref{Equ:ConstrainedGP} and \eqref{Equ:ConstrainedSNP} and provide the desired properties in the Theorems \ref{Thm:ConstrainedGP} and \ref{Thm:ConstrainedSP}. Moreover, we show that under certain extra assumptions the bound obtained by the new program \eqref{Equ:ConstrainedGP} equals the optimal bound $s(f,\mathbf{g})$ from the previous section; see Theorem \ref{Thm:ConstrainedGPequality}.\\

Let all notation regarding $G(\boldsymbol{\mu})$ be given as in Section \ref{SubSec:PrelimConstraint}. Assume that we have for each $0 \leq i \leq s$
\begin{eqnarray*}
 \struc{g_i} & = & \sum_{\boldsymbol{\beta} \in A_i} g_{i,\boldsymbol{\beta}} \cdot \mathbf{x}^{\boldsymbol{\beta}}
\end{eqnarray*}
with $g_{i,\boldsymbol{\beta}} \in \R$. We have $\Delta(A_i) \subseteq \Delta(A)$ and hence write 
\begin{eqnarray*}
 g_i & = & \sum_{j = 0}^r g_{i,\boldsymbol{\alp}(j)} \mathbf{x}^{\boldsymbol{\alp}(j)} + \sum_{\boldsymbol{\beta} \in \Delta(A)} g_{i,\boldsymbol{\beta}} \mathbf{x}^{\boldsymbol{\beta}}
\end{eqnarray*}
and set $g_{i,\boldsymbol{\alp}(j)} = 0$ for all $\boldsymbol{\alp}(j) \in V(A) \setminus A_i$ and $g_{i,\boldsymbol{\beta}} = 0$ for all $\boldsymbol{\beta} \in \Delta(A) \setminus A_i$. We remark that three cases can occur for $\boldsymbol{\beta} \in \Delta(A) \cap A_i$:
\begin{enumerate}
 \item $-g_{i,\boldsymbol{\beta}} \mathbf{x}^{\boldsymbol{\beta}}$ is not a monomial square. Then we have $\boldsymbol{\beta} \in \Delta(G)$.
 \item $-g_{i,\boldsymbol{\beta}} \mathbf{x}^{\boldsymbol{\beta}}$ is a monomial square, but there exists another $g_l$ such that $-g_{l,\boldsymbol{\beta}} \mathbf{x}^{\boldsymbol{\beta}}$ is not a monomial square. Then we have $\boldsymbol{\beta} \in \Delta(G)$.
 \item $-g_{i,\boldsymbol{\beta}} \mathbf{x}^{\boldsymbol{\beta}}$ is a monomial square, and there exists no other $g_l$ such that $-g_{l,\boldsymbol{\beta}} \mathbf{x}^{\boldsymbol{\beta}}$ is not a monomial square. Then we have $\boldsymbol{\beta} \notin \Delta(G)$.
\end{enumerate}

Sums of monomial squares as described in case (3) are ignored in our program \eqref{Equ:Constrained}.
Hence, we can also ignore this case here. We investigate the other two cases in detail now. As already mentioned in Section \ref{SubSec:PrelimConstraint} we can interpret the coefficients $G(\boldsymbol{\mu})_{\boldsymbol{\alp}(j)}$ and $G(\boldsymbol{\mu})_{\boldsymbol{\beta}}$ as linear forms in $\boldsymbol{\mu}$ since we have for all $j = 0,\ldots,r$
\begin{eqnarray*}
 & & \struc{G(\boldsymbol{\mu})_{\boldsymbol{\alp}(j)}} \ = \ - \sum_{i = 0}^s \mu_i \cdot g_{i,\boldsymbol{\alp}(j)} \ \text{ and } \ \struc{G(\boldsymbol{\mu})_{\boldsymbol{\beta}}} \ = \  - \sum_{i = 0}^s \mu_i \cdot g_{i,\boldsymbol{\beta}}.
\end{eqnarray*}

We decompose every $G(\boldsymbol{\mu})_{\boldsymbol{\beta}}$ into a positive and a negative part such that $G(\boldsymbol{\mu})_{\boldsymbol{\beta}}=G(\boldsymbol{\mu})_{\boldsymbol{\beta}}^+-G(\boldsymbol{\mu})_{\boldsymbol{\beta}}^-$, where

\begin{eqnarray}
& & \struc{G(\boldsymbol{\mu})_{\boldsymbol{\beta}}^-} \ = \ \sum_{g_{i,\boldsymbol{\beta}}>0} \mu_i \cdot g_{i,\boldsymbol{\beta}} \ \text{ and } \ \struc{G(\boldsymbol{\mu})_{\boldsymbol{\beta}}^+} \ = \ - \sum_{g_{i,\boldsymbol{\beta}}<0} \mu_i \cdot g_{i,\boldsymbol{\beta}}. \label{Equ:Gplusminus}
\end{eqnarray}

This decomposition is \textit{independent} of the choice of $\boldsymbol{\mu}$ in the sense that no $g_{i,\boldsymbol{\beta}}$ can be a summand of \textit{both} $G(\boldsymbol{\mu})_{\boldsymbol{\beta}}^+$ and $G(\boldsymbol{\mu})_{\boldsymbol{\beta}}^-$ for different choices of $\boldsymbol{\mu}$ since $\boldsymbol{\mu} \in \R^{s}_{\geq 0}$. The key idea is to redefine the constraint $b_{\boldsymbol{\beta}} \geq |G(\boldsymbol{\mu})_{\boldsymbol{\beta}}|$ by a new constraint $b_{\boldsymbol{\beta}} \geq \max\{G(\boldsymbol{\mu})_{\boldsymbol{\beta}}^+,G(\boldsymbol{\mu})_{\boldsymbol{\beta}}^-\}$. Let $R$ be defined as in Section \ref{SubSec:PrelimConstraint} and let $\struc{g_{i,\boldsymbol{\alp}(0)}^+}= \max\{g_{i,\boldsymbol{\alp}(0)},0\}$, i.e., we only consider the terms with exponents $\boldsymbol{\alp}(0)$ which are positive in the $g_i$ and thus negative in $G(\boldsymbol{\mu})$. We redefine $p$ as
\begin{eqnarray*}
 & & \struc{p(\boldsymbol{\mu},\{(\mathbf{a}_{\boldsymbol{\beta}},b_{\boldsymbol{\beta}}) \ : \ \boldsymbol{\beta} \in \Delta(G)\})} \ = \ \\
 & & \sum_{i=1}^{s} \mu_i g_{i,\boldsymbol{\alp}(0)}^+ + 
 \sum_{\substack{\boldsymbol{\beta} \in \Delta(G) \\ \lambda_0^{(\boldsymbol{\beta})} \neq 0}}^{}\lambda_0^{(\boldsymbol{\beta})} \cdot b_{\boldsymbol{\beta}}^{\frac{1}{\lambda_0^{(\boldsymbol{\beta})}}}\cdot \prod_{\substack{j\in {\rm nz}(\boldsymbol{\beta}) \\ j\geq 1}} \left(\frac{\lambda_j^{(\boldsymbol{\beta})}}{a_{\boldsymbol{\beta},j}}\right)^{\frac{\lambda_j^{(\boldsymbol{\beta})}}{\lambda_0^{(\boldsymbol{\beta})}}}.
\end{eqnarray*}

We consider the following optimization problem in the variables $\mu_1,\ldots,\mu_s$ and \\ $a_{\boldsymbol{\beta},1},\ldots,a_{\boldsymbol{\beta},r},b_{\boldsymbol{\beta}}$ for every $\boldsymbol{\beta} \in \Delta(G)$.

\begin{eqnarray}
& & \begin{cases}
\text{minimize} &
p(\boldsymbol{\mu},\{(\mathbf{a}_{\boldsymbol{\beta}},b_{\boldsymbol{\beta}}) \ : \ \boldsymbol{\beta} \in \Delta(G)\})
\text { over the subset of } R
\\
& \\
\text{defined by:} & 
\begin{array}{cl}
 (1) & \sum\limits_{\boldsymbol{\beta}\in\Delta(G)} a_{\boldsymbol{\beta},j} \, \leq \, G(\boldsymbol{\mu})_{\boldsymbol{\alp}(j)} \ \text{ for all } \ 1 \leq j \leq r,  \\
 (2) & \prod\limits_{j\in {\rm nz}(\boldsymbol{\beta})} \left(\frac{a_{\boldsymbol{\beta},j}}{\lambda_j^{(\boldsymbol{\beta})}}\right)^{\lambda_j^{(\boldsymbol{\beta})}} \, \geq \,  b_{\boldsymbol{\beta}},
 \ \text{ for all } \ \boldsymbol{\beta} \in \Delta(G) \ \text{with} \ \lambda_0^{(\boldsymbol{\beta})}=0, \\
(3) & G(\boldsymbol{\mu})_{\boldsymbol{\beta}}^+ \leq b_{\boldsymbol{\beta}} \ \text{ for all } \ \boldsymbol{\beta}\in \Delta(G), \ \text{ and } \\
(4) & G(\boldsymbol{\mu})_{\boldsymbol{\beta}}^- \leq b_{\boldsymbol{\beta}} \ \text{ for all } \ \boldsymbol{\beta}\in \Delta(G).  \\
\end{array}
\end{cases}
\label{Equ:ConstrainedGP}
\end{eqnarray}

This problem is, by condition (1), feasible only for choices of $\boldsymbol{\mu}$ such that $G(\boldsymbol{\mu})_{\boldsymbol{\alp}(j)} > 0$ for all $\boldsymbol{\alp}(j)$ since all $a_{\boldsymbol{\beta},j}$ are strictly positive. We set the output as $-\infty$ in all other cases. Indeed, with some additional assumptions the program \eqref{Equ:ConstrainedGP} is a geometric program. Moreover, it is  a relaxation of the program \eqref{Equ:Constrained}.

\begin{thm}
Assume that for every $1 \leq j \leq r$ the form $G(\boldsymbol{\mu})_{\boldsymbol{\alp}(j)}  =  - \sum_{i = 0}^s \mu_i \cdot g_{i,\boldsymbol{\alp}(j)}$ has exactly one strictly positive term, i.e. there exists exactly one strictly negative $g_{i,\boldsymbol{\alp}(j)}$. Then the optimization problem \eqref{Equ:ConstrainedGP} restricted to $\boldsymbol{\mu} \in (0,\infty)^s$ is a geometric program.  Assume that $\struc{\gamma_{\rm sonc}}$ denotes the optimal value of \eqref{Equ:ConstrainedGP} and $\struc{\gamma}$ denotes the optimal value of \eqref{Equ:Constrained}. Then we have
\begin{eqnarray*}
 f_{\boldsymbol{\alp}(0)} - \gamma_{\rm sonc} \ \leq \ f_{\boldsymbol{\alp}(0)} - \gamma \ \leq \ s(f,\mathbf{g}).
\end{eqnarray*}
\label{Thm:ConstrainedGP}
\end{thm}

The typical choice for $\boldsymbol{\alp}(0)$ is the origin which yields a lower bound for $f$ to be nonnegative on $K$ with the inequality \eqref{Equ:fKStern}:

\begin{cor}
Let all assumptions be as in Theorem $\ref{Thm:ConstrainedGP}$. If $\boldsymbol{\alp}(0)$ is the origin, then we have
\begin{eqnarray*}
 f_{\mathbf{0}} - \gamma_{\rm sonc} \ \leq \ f_{\mathbf{0}} - \gamma \ \leq \ s(f,\mathbf{g}) \ \leq \ f^*_K.
\end{eqnarray*}
\end{cor}

\begin{proof}(Theorem \ref{Thm:ConstrainedGP})
If we restrict ourselves to $\boldsymbol{\mu} \in (0,\infty)^s$, then all functions involved in \eqref{Equ:ConstrainedGP} depend on variables in $\R_{> 0}$.  By assumption every $G(\boldsymbol{\mu})_{\boldsymbol{\alp}(j)}$ has exactly one strictly positive term. Thus, we can express constraint (1) as
\begin{eqnarray*}
 \frac{\sum\limits_{\boldsymbol{\beta}\in\Delta(G)} a_{\boldsymbol{\beta},j} + G(\boldsymbol{\mu})_{\boldsymbol{\alp}(j)}^-}{G(\boldsymbol{\mu})_{\boldsymbol{\alp}(j)}^+} & \leq & 1,
\end{eqnarray*}
with $ G(\boldsymbol{\mu})_{\boldsymbol{\alp}(j)}^-$ and $G(\boldsymbol{\mu})_{\boldsymbol{\alp}(j)}^+$ defined analogously as in \eqref{Equ:Gplusminus}. Since $G(\boldsymbol{\mu})_{\boldsymbol{\alp}(j)}^+$ is a monomial the left hand side is a posynomial in $\boldsymbol{\mu}$ and $\mathbf{x}$. The constraints (2) -- (4) are posynomial constraints in the sense of Definition \ref{Def:GP} of a geometric program. The function $p$ is also a posynomial since all terms are nonnegative by construction and all exponents are rational. Moreover, every $b_{\boldsymbol{\beta}}$ in \eqref{Equ:ConstrainedGP} has to be greater or equal than the corresponding $b_{\boldsymbol{\beta}}$ in \eqref{Equ:Constrained} because $\max\{a,b\} \geq |a-b|$ for all $a,b \in \R\setminus\{0\}$. Since furthermore $g_{i,\boldsymbol{\alp}(0)}^+ \geq g_{i,\boldsymbol{\alp}(0)}$ it follows that $\gamma_{\rm sonc} \leq \gamma$ by the definitions of \eqref{Equ:ConstrainedGP} and \eqref{Equ:Constrained}. The last inequality follows from Theorem \ref{Thm:GPOldConstrainedCase}.
\end{proof}

One expects the programs \eqref{Equ:Constrained} and \eqref{Equ:ConstrainedGP} to have a similar optimal value if, for example, $g_{i,\boldsymbol{\alp}(0)}\geq 0$ for most $i=1,\ldots,s$ and if one $G(\boldsymbol{\mu})^+_{\boldsymbol{\beta}}, G(\boldsymbol{\mu})^-_{\boldsymbol{\beta}}$ is identically zero for most $\boldsymbol{\beta} \in \Delta(G)$. Note that one of $G(\boldsymbol{\mu})^+_{\boldsymbol{\beta}}, G(\boldsymbol{\mu})^-_{\boldsymbol{\beta}}$ is zero if and only if $\max\{G(\boldsymbol{\mu})^+_{\boldsymbol{\beta}},G(\boldsymbol{\mu})^-_{\boldsymbol{\beta}}\}=|G(\boldsymbol{\mu})^+_{\boldsymbol{\beta}}-G(\boldsymbol{\mu})^-_{\boldsymbol{\beta}}|=|G(\boldsymbol{\mu})_{\boldsymbol{\beta}}|$ if and only if the $g_{i,\boldsymbol{\beta}}$ are all $\geq 0$ or all $\leq 0$ for $i=0,\ldots,s$.\\

We give an example to demonstrate how a given constrained polynomial optimization problem can be translated into the geometric program \eqref{Equ:ConstrainedGP}. In Section \ref{Sec:ExamplesConstrained}, we provide several further examples including actual computations of infima using the GP-solver \textsc{CVX}.
	
\begin{example}
 Let $f= 1+2x^2y^4+\frac{1}{2}x^3y^2 $ and $g_1= \frac{1}{3} - x^6y^2$. From these two polynomials we obtain a function
\[
 G(\mu) \ = \ \left(1-\frac{1}{3}\mu\right)+2x^2y^4 + \mu x^6y^2 + \frac{1}{2} x^3y^2.
\]
For $G(\mu)$ to be an ST-polynomial, we have to choose $\mu \in (0,3)$. Here, the vertices of $\New(G(\mu))$ are $\boldsymbol{\alp}(0)=(0,0), \boldsymbol{\alp}(1)=(2,4), \boldsymbol{\alp}(2)=(6,2)$, and we have $\Delta(G)=\{\boldsymbol{\beta}\}=\{(3,2)\}$. Thus, we introduce $4$ variables $(a_{\boldsymbol{\beta},1},a_{\boldsymbol{\beta},2},b_{\boldsymbol{\beta}},\mu)$.
First, we compute the barycentric coordinates of $\boldsymbol{\beta}$ and get 
\[
 \lambda_0^{(\boldsymbol{\beta})} \ = \ \frac{3}{10}, \quad  \lambda_1^{(\boldsymbol{\beta})} \ = \ \frac{3}{10},\quad \lambda_2^{(\boldsymbol{\beta})} \ = \ \frac{2}{5} .
\]
We match the coefficients of $G(\mu)$ with the vertices $\boldsymbol{\alp}(j)$:
\begin{itemize}
 \item  $g_{1,\boldsymbol{\alp}(0)}^+=\max\{\frac{1}{3},0\}=\frac{1}{3}$ , 
 \item  $G(\mu)_{\boldsymbol{\alp}(1)}=2$, $G(\mu)_{\boldsymbol{\alp}(2)}=\mu$ , 
 \item  $G(\mu)_{\boldsymbol{\beta}}^+=\frac{1}{2}$, $G(\mu)_{\boldsymbol{\beta}}^-$ does not exist.
\end{itemize}
Hence, Program  \eqref{Equ:ConstrainedGP} is of the form:
\begin{eqnarray*}
& & \inf\left\{\frac{1}{3} \mu + \frac{3}{10} \cdot b_{\boldsymbol{\beta}}^{\frac{10}{3}} \cdot \left(\frac{3}{10}\right)^1 \cdot \left(\frac{2}{5}\right)^{\frac{4}{3}} \cdot (a_{\boldsymbol{\beta},1})^{-1} \cdot (a_{\boldsymbol{\beta},2})^{-\frac{4}{3}} \right\}
\end{eqnarray*}
such that:
\begin{enumerate}
 \item $a_{\boldsymbol{\beta},1} \leq 2 ,\; a_{\boldsymbol{\beta},2} \leq \mu$.
 \item The second constraint does not appear, because we do not have $\lambda_0^{(\boldsymbol{\beta})}=0$.
 \item $\frac{1}{2} \leq b_{\boldsymbol{\beta}}$.
 \item The fourth constraint does not appear, because we do not have a $G(\mu)_{\boldsymbol{\beta}}^-$.
\end{enumerate}
\endexa
\end{example}

\bigskip

In what follows, we extend Theorem \ref{Thm:GPOldConstrainedCase} by reformulating the program \eqref{Equ:ConstrainedGP} such that it is \textit{always} applicable. On the one hand, the new program is only a signomial program instead of a geometric program in general.
 On the other hand, the reformulated program covers the missing cases of Theorem \ref{Thm:ConstrainedGP} and also yields better bounds than the corresponding geometric program \eqref{Equ:ConstrainedGP} in general. Let 
\begin{eqnarray*}
 \struc{q(\boldsymbol{\mu},\{(\mathbf{a}_{\boldsymbol{\beta}},c_{\boldsymbol{\beta}}) \ : \ \boldsymbol{\beta} \in \Delta(G)\})} & = &
 \sum_{i=1}^{s} \mu_ig_{i,\boldsymbol{\alp}(0)} + 
 \sum_{\substack{\boldsymbol{\beta} \in \Delta(G) \\ \lambda_0^{(\boldsymbol{\beta})} \neq 0}}^{}\lambda_0^{(\boldsymbol{\beta})} \cdot c_{\boldsymbol{\beta}}^{\frac{1}{\lambda_0^{(\boldsymbol{\beta})}}}\cdot \prod_{\substack{j\in {\rm nz}(\boldsymbol{\beta}) \\ j\geq 1}} \left(\frac{\lambda_j^{(\boldsymbol{\beta})}}{a_{\boldsymbol{\beta},j}}\right)^{\frac{\lambda_j^{(\boldsymbol{\beta})}}{\lambda_0^{(\boldsymbol{\beta})}}}.
\end{eqnarray*}

We consider the following program.

\begin{eqnarray}
& & \begin{cases}
\text{minimize} &
q(\boldsymbol{\mu},\{(\mathbf{a}_{\boldsymbol{\beta}},c_{\boldsymbol{\beta})} \ : \ \boldsymbol{\beta} \in \Delta(G)\})
\text { over the subset of } R
\\
& \\
\text{defined by:} & 
\begin{array}{cl}
 (1) & \sum\limits_{\boldsymbol{\beta}\in\Delta(G)} a_{\boldsymbol{\beta},j} \, \leq \, G(\boldsymbol{\mu})_{\boldsymbol{\alp}(j)} \ \text{ for all } \ 1 \leq j \leq r,  \\
 (2) & \prod\limits_{j\in {\rm nz}(\boldsymbol{\beta})} \left(\frac{a_{\boldsymbol{\beta},j}}{\lambda_j^{(\boldsymbol{\beta})}}\right)^{\lambda_j^{(\boldsymbol{\beta})}} \, \geq \,  c_{\boldsymbol{\beta}}
 \ \text{ for all } \ \boldsymbol{\beta} \in \Delta(G) \ \text{with} \ \lambda_0^{(\boldsymbol{\beta})}=0, \\
(3) & G(\boldsymbol{\mu})_{\boldsymbol{\beta}}^+ - G(\boldsymbol{\mu})_{\boldsymbol{\beta}}^- \leq c_{\boldsymbol{\beta}} \ \text{ for all } \ \boldsymbol{\beta} \in \Delta(G), \ \text{ and } \\
(4) & G(\boldsymbol{\mu})_{\boldsymbol{\beta}}^- - G(\boldsymbol{\mu})_{\boldsymbol{\beta}}^+ \leq c_{\boldsymbol{\beta}} \ \text{ for all } \ \boldsymbol{\beta} \in \Delta(G).  \\
\end{array}
\end{cases}
\label{Equ:ConstrainedSNP}
\end{eqnarray}

The key difference between this program and \eqref{Equ:ConstrainedGP} is that 
$$c_{\boldsymbol{\beta}} \ \geq \ \max\{G(\boldsymbol{\mu})_{\boldsymbol{\beta}}^+ - G(\boldsymbol{\mu})_{\boldsymbol{\beta}}^-, G(\boldsymbol{\mu})_{\boldsymbol{\beta}}^- - G(\boldsymbol{\mu})_{\boldsymbol{\beta}}^+\} \ = \ |G(\boldsymbol{\mu})_{\boldsymbol{\beta}}|.$$
We obtain the following statement.

\begin{thm}
The optimization problem \eqref{Equ:ConstrainedSNP} restricted to $\boldsymbol{\mu} \in (0,\infty)^s$ is a signomial program. Assume that $\struc{\gamma_{\rm snp}}$ denotes the optimal value of \eqref{Equ:ConstrainedSNP} and $\gamma_{\rm sonc},\gamma$ denote the optimal values of \eqref{Equ:ConstrainedGP} and \eqref{Equ:Constrained} as before. Then we have
\begin{eqnarray*}
 f_{\boldsymbol{\alp}(0)} - \gamma_{\rm sonc} \ \leq \ f_{\boldsymbol{\alp}(0)} - \gamma_{\rm snp} \ \leq \ f_{\boldsymbol{\alp}(0)} - \gamma \ \leq \ s(f,\mathbf{g}).
\end{eqnarray*}
Particularly, we have $\gamma_{\rm snp} = \gamma$ if the program \eqref{Equ:Constrained} attains its optimal value for $\boldsymbol{\mu} \in (0,\infty)^s$.
\label{Thm:ConstrainedSP}
\end{thm}

\begin{proof}
The proof is analogue to the proof of Theorem \ref{Thm:ConstrainedGP}. The only difference is that certain terms can have a negative sign now and hence posynomials then become signomials. The statement follows with the definition of a signomial program; see Section \ref{SubSec:PrelimGP}.
\end{proof}

Finally, we show that if we strengthen the assumptions in Theorem \ref{Thm:ConstrainedGP}, then, the output $f_{\boldsymbol{\alp}(0)} - \gamma_{\rm sonc}$ of \eqref{Equ:ConstrainedGP} equals the output $f_{\boldsymbol{\alp}(0)} - \gamma$ of \eqref{Equ:Constrained} and particularly the bound $s(f,\mathbf{g})$.

\begin{thm}
 Assume that for every $1 \leq j \leq r$ the form $G(\boldsymbol{\mu})_{\boldsymbol{\alp}(j)}  =  - \sum_{i = 0}^s \mu_i \cdot g_{i,\boldsymbol{\alp}(j)}$ has exactly one strictly positive term. Furthermore, assume that $g_{i,\boldsymbol{\alp}(0)}\geq 0$ for all $i=1,\ldots,s$, and that $\Delta(A) \cap A_i \cap A_l = \emptyset$ for all $0\leq i<l\leq s$. Let \struc{$\gamma$} be the optimal value of the program \eqref{Equ:Constrained}. If the optimal value $s(f,\mathbf{g}) = \sup\{G(\boldsymbol{\mu})_{\rm sonc} : \boldsymbol{\mu} \in \R^{s}_{\geq 0}\}$ is attained for some $\boldsymbol{\mu} \in (0,\infty)^s$, then $f_{\boldsymbol{\alp}(0)} - \gamma_{\rm sonc} = f_{\boldsymbol{\alp}(0)} - \gamma = s(f,\mathbf{g})$, where, as before, $\struc{\gamma_{\rm sonc}}$ denotes the optimal value of \eqref{Equ:ConstrainedGP}.
 \label{Thm:ConstrainedGPequality}
\end{thm} 

Note that the condition $\Delta(A) \cap A_i \cap A_l=\emptyset$ is satisfied if the supports of $g_i$ and $g_l$ differ in all elements that are not vertices of $\New(G(\boldsymbol{\mu}))$.

\begin{proof}
The assumption $\Delta(A) \cap A_i \cap A_l=\emptyset$ for all $0\leq i<l\leq s$ implies for every $\boldsymbol{\beta} \in \Delta(G)$ that $G(\boldsymbol{\mu})_{\boldsymbol{\beta}}=-\sum_{i=0}^s \mu_i \cdot g_{i,\boldsymbol{\beta}}=-\mu_k\cdot g_{k,\boldsymbol{\beta}}$, for some $k \in [0,s]$. Therefore, we have for every $\boldsymbol{\beta} \in \Delta(G)$ that
\begin{eqnarray*}
 \max\{G(\boldsymbol{\mu})^+_{\boldsymbol{\beta}},G(\boldsymbol{\mu})^-_{\boldsymbol{\beta}}\} \ = \ |\mu_k\cdot g_{k,\boldsymbol{\beta}}| \ = \ |G(\boldsymbol{\mu})_{\boldsymbol{\beta}}|.
\end{eqnarray*}
Furthermore, we have $g_{i,\boldsymbol{\alp}(0)} \geq 0$ for all $i=1,\ldots,s$ by assumption and thus we obtain $\sum_{i = 1}^s \mu_i g_{i,\boldsymbol{\alp}(0)} = \sum_{i = 1}^s \mu_i g_{i,\boldsymbol{\alp}(0)}^+$. Hence, the two programs \eqref{Equ:Constrained} and \eqref{Equ:ConstrainedGP} coincide.

By assumption, every $G(\boldsymbol{\mu})_{\boldsymbol{\alp}(j)}$ consists of exactly one positive term. Therefore, \eqref{Equ:ConstrainedGP} is a GP by Theorem \ref{Thm:ConstrainedGP}. Considering Theorem \ref{Thm:ConstrainedGP} it suffices to show the inequality $f_{\boldsymbol{\alp}(0)}- \gamma_{\rm sonc} \geq s(f,\mathbf{g})$ for $f_{\boldsymbol{\alp}(0)} - \gamma_{\rm sonc} = f_{\boldsymbol{\alp}(0)} - \gamma = s(f,\mathbf{g})$ to hold.
Let $\boldsymbol{\mu}^* \in (0,\infty)^s$ be such that $G(\boldsymbol{\mu}^*)_{\rm sonc} = s(f,\mathbf{g})$. By Corollary \ref{Cor:GP} $G(\boldsymbol{\mu}^*)_{\rm sonc}$ is given by a feasible point $(a_{\boldsymbol{\beta},1},\ldots,a_{\boldsymbol{\beta},r})$ of the program 

 \begin{eqnarray*}
 & & \begin{cases}
\text{minimize} &
 \sum\limits_{\substack{\boldsymbol{\beta} \in \Delta(G) \\ \lambda_0^{(\boldsymbol{\beta})} \neq 0}}^{}\lambda_0^{(\boldsymbol{\beta})} \cdot |\mu^*_k\cdot g_{k,\boldsymbol{\beta}}|^{\frac{1}{\lambda_0^{(\boldsymbol{\beta})}}}\cdot \prod\limits_{\substack{j\in {\rm nz}(\boldsymbol{\beta}) \\ j\geq 1}} \left(\frac{\lambda_j^{(\boldsymbol{\beta})}}{a_{\boldsymbol{\beta},j}}\right)^{\frac{\lambda_j^{(\boldsymbol{\beta})}}{\lambda_0^{(\boldsymbol{\beta})}}}
\text{ over the subset } R' \text{ of } R
 \\
 & \\
 \text{defined by:} & 
 \begin{array}{cl}
  (1) & \sum\limits_{\boldsymbol{\beta} \in\Delta(G)} a_{\boldsymbol{\beta},j} \, \leq \, G(\boldsymbol{\mu}^*)_{\boldsymbol{\alp}(j)} \ \text{ for all } \ 1 \leq j \leq r, \text{ and } \\
 (2) & \prod\limits_{\substack{j\in {\rm nz}(\boldsymbol{\beta})}} \left(\frac{a_{\boldsymbol{\beta},j}}{\lambda_j^{(\boldsymbol{\beta})}}\right)^{\lambda_j^{(\boldsymbol{\beta})}} \,\geq \, |\mu^*_k\cdot g_{k,\boldsymbol{\beta}}|\ \text{ for all } \ \boldsymbol{\beta} \in \Delta(G) \ \text{with} \ \lambda_0^{(\boldsymbol{\beta})}=0.  \\ 
 \end{array}
 \end{cases}
\end{eqnarray*}
 
Then every $(a_{\boldsymbol{\beta},1},\ldots,a_{\boldsymbol{\beta},r},b_{\boldsymbol{\beta}},\boldsymbol{\mu}^*)$ with $b_{\boldsymbol{\beta}}\geq |\mu^*_k\cdot g_{k,\boldsymbol{\beta}}|$ for all $\boldsymbol{\beta} \in \Delta(G)$ is a feasible point of \eqref{Equ:ConstrainedGP}. Furthermore, 

 \begin{multline*}
f_{\boldsymbol{\alp}(0)}-\sum_{i=1}^s \mu^*_i g_{i,\boldsymbol{\alp}(0)}^+ - \sum_{\substack{\boldsymbol{\beta} \in \Delta(G) \\ \lambda_0^{(\boldsymbol{\beta})} \neq 0}}^{}\lambda_0^{(\boldsymbol{\beta})} \cdot b_{\boldsymbol{\beta}}^{\frac{1}{\lambda_0^{(\boldsymbol{\beta})}}}\cdot \prod_{\substack{j\in {\rm nz}(\boldsymbol{\beta}) \\ j\geq 1}} \left(\frac{\lambda_j^{(\boldsymbol{\beta})}}{a_{\boldsymbol{\beta},j}}\right)^{\frac{\lambda_j^{(\boldsymbol{\beta})}}{\lambda_0^{(\boldsymbol{\beta})}}}\\ 
= G(\boldsymbol{\mu}^*)(0) - \sum_{\substack{\boldsymbol{\beta} \in \Delta(G) \\ \lambda_0^{(\boldsymbol{\beta})} \neq 0}}^{}\lambda_0^{(\boldsymbol{\beta})} \cdot b_{\boldsymbol{\beta}}^{\frac{1}{\lambda_0^{(\boldsymbol{\beta})}}}\cdot \prod_{\substack{j\in {\rm nz}(\boldsymbol{\beta}) \\ j\geq 1}} \left(\frac{\lambda_j^{(\boldsymbol{\beta})}}{a_{\boldsymbol{\beta},j}}\right)^{\frac{\lambda_j^{(\boldsymbol{\beta})}}{\lambda_0^{(\boldsymbol{\beta})}}}\\ 
\geq G(\boldsymbol{\mu}^*)(0) - \sum_{\substack{\boldsymbol{\beta} \in \Delta(G) \\ \lambda_0^{(\boldsymbol{\beta})} \neq 0}}^{}\lambda_0^{(\boldsymbol{\beta})} \cdot |\mu^*_k\cdot g_{k,\boldsymbol{\beta}}|^{\frac{1}{\lambda_0^{(\boldsymbol{\beta})}}}\cdot \prod_{\substack{j\in {\rm nz}(\boldsymbol{\beta}) \\ j\geq 1}} \left(\frac{\lambda_j^{(\boldsymbol{\beta})}}{a_{\boldsymbol{\beta},j}}\right)^{\frac{\lambda_j^{(\boldsymbol{\beta})}}{\lambda_0^{(\boldsymbol{\beta})}}}.
\end{multline*}

Hence, $f_{\boldsymbol{\alp}(0)} - \gamma_{\rm sonc}\geq G(\boldsymbol{\mu}^*)_{\rm sonc} = s(f,\mathbf{g})$.

\end{proof}

\section{Examples for Constrained Optimization via Geometric Programming and a Comparison to Lasserre Relaxations}
\label{Sec:ExamplesConstrained}

We consider constrained polynomial optimization problems of the form
\begin{eqnarray*}
f_K^* \ = \ \inf_{\mathbf{x} \in K}^{}f(\mathbf{x}),
\end{eqnarray*}
where $K$ is a basic closed semialgebraic set defined by $g_1,\dots,g_s \geq 0$. One of the main results in \cite{Iliman:deWolff:GP} is the observation that lower bounds for global optimization problems arising from SONCs via GP can not only be computed faster, but also be better than the bounds obtained by SOS via SDP. Here, we show that competitive bounds arising from SONC via GP can also be obtained for constrained problems. Particularly, if $2d$ is the maximal total degree of $f$ and $g_1,\ldots,g_s$, then the bound given by the $d$-th Lasserre relaxation is not necessarily as good as our optimal solution, which is in contrast to the bounds obtained by Ghasemi and Marshall; see Example \ref{Exa:Constrained:GhasemiMarshallCase} for further details. Moreover, we provide runtimes for rescaled version of our examples demonstrating that the runtime of the GP approach is not sensitive to increasing the degree of a given problem. This has been observed by several authors in the past, e.g. \cite{Boyd:Kim:Vandenberghe:TutorialOnGP,Ghasemi:Marshall:GP, Ghasemi:Marshall:GP:Semialgebraic}, and is in contrast to the runtime of SDPs.\\

Let \struc{$\Sigma_{}^{}\mathbb R[\mathbf{x}]^2$} denote the set of \struc{\textit{$n$-variate sums of squares}}. We consider the \struc{$d$-th \emph{Lasserre relaxation}} \cite{Lasserre:NonnegativeApplications}
\begin{eqnarray}
 & & \struc{f_{\rm sos}^{(d)}} \ = \ \sup\left\{r \ : \ f - r = \sum_{i=0}^s \sigma_ig_i,\,\,\sigma_i\in\Sigma_{}^{}\mathbb R[\mathbf{x}]^2, \, \ g_0 = 1, \ \deg(\sigma_ig_i)\leq 2d\right\},
 \label{Equ:LasserreHierarchie}
\end{eqnarray} 
where 
$d\geq \max\lf\{\ceil*{\deg(f)/2}, \max_{1\leq i\leq s}\{\ceil*{\deg(g_i)/2}\} \ri\}$.\\

In what follows we provide several examples comparing Lasserre relaxation using the \textsc{Matlab} SDP solver  \textsc{Gloptipoly} \cite{gloptipoly} to our approach given in program \eqref{Equ:ConstrainedGP} using the \textsc{Matlab} GP solver \textsc{CVX} \cite{Boyd:Grant:CVX2,Boyd:Grant:Ye:CVX}. In every example in this section we optimize with respect to the constant term when applying program \eqref{Equ:ConstrainedGP}.

\begin{example}
  Let $f = 1 + x^4y^2+ x^2y^4 - 3x^2y^2$ be the Motzkin polynomial and $g_1 = x^3y^2$. Then 
  $$K \ = \ \{(x,y)\in\mathbb R^2 \ : \ x\geq 0 \text{ or } y = 0\}.$$
  Since $f$ is globally nonnegative and has two zeros $(1,1), (1,-1)$ on $K$, e.g. \cite{Reznick:SurveyHilbert17th}, we have $f_K^* = 0$. We consider the third Lasserre relaxation and obtain
  $$f_{\rm sos}^{(3)}  =  \sup\left\{r \ : \ f - r = \sigma_0 + \sigma_1\cdot g_1, \,\,\sigma_0,\sigma_1\in\Sigma_{}^{}\mathbb R[\mathbf{x}]^2, \deg(\sigma_0)\leq 6,\deg(\sigma_1g_1)\leq 6\right\}  =  -\infty,$$ since the problem is infeasible. Note that $K$ is unbounded. Hence, it is not necessarily the case that $f_{\rm sos}^{(d)} > -\infty$ for sufficiently high relaxation order $d$. Here, using \textsc{Gloptipoly}, one can find that $f_{\rm sos}^{(7)} = 0 = f_K^*$.
  
Now, we consider $s(f,g_1) = \sup\{G(\mu)_{\rm sonc} : \mu \in \R_{\geq 0}\} \ \leq \ f_K^*$ where $G(\mu) = f - \mu g_1$ with $\mu \geq 0$. Note that $\New(G(\mu))$ is a simplex for every choice of $\mu$. In particular, for $\mu = 0$ we have that $G(\mu)_{\rm sonc} = f_{\rm sonc} = 0$, since the Motzkin polynomial is a SONC polynomial; see Section \ref{SubSec:PrelimSONC} and also \cite{Iliman:deWolff:Circuits}. It follows that
$$-\infty \ = \ f_{\rm sos}^{(3)} \ < \ s(f,g_1) \ = \ 0 \ = \ f_K^*.$$ 
Hence, $s(f,g_1)$ yields the exact solution compared to the Lasserre relaxation. This is in sharp contrast to the geometric programming approach proposed in \cite{Ghasemi:Marshall:GP:Semialgebraic} where $f_{\rm sos}^{(d)} \geq s(f,\mathbf{g})$ holds in general.
\endexa
\end{example}

Note that there are techniques which handle polynomial optimization problems over unbounded feasible sets better than the SOS/SDP method, e.g. using the gradient ideal \cite{Nie:Demmel:Sturmfels:GradientIdeal} or using gradient tentacles \cite{Schweighofer:GradientTentacles}. Here, we restrict ourselves, however, to a comparison to SDP based methods.

\begin{example}
Let $f = 1 + x^4y^2 + xy$ and $g_1 = \frac{1}{2} + x^2y^4 - x^2y^6$. The feasible set $K$ is a non-compact set depicted in Figure \ref{Fig:FeasibleSet1}. Using \textsc{Gloptipoly}, one can check that $-\infty = f_{\rm sos}^{(4)}$ and the optimal solution is given for $d = 8$ with $f_{\rm sos}^{(8)} \approx 0.4474$. In this case one can extract the minimizers $(-0.557,1.2715)$ and $(0.557,-1.2715)$. 

We compare the results to our approach via geometric programming instead of Lasserre relaxations. From $f$ and $g_1$ we get $G(\mu)=(1-\frac{1}{2}\mu)+x^4y^2+\mu x^2y^6+xy-\mu x^2y^4$. Note that $\New(G(\mu))$ is a two dimensional simplex if $\mu\notin\{0,2\}$. Then, we have $\Delta(G)=\{\boldsymbol{\beta}, \tilde{\boldsymbol{\beta}}\}=\{(1,1),(2,4)\}$. Hence, we introduce the variables $(a_{\boldsymbol{\beta},1}, a_{\boldsymbol{\beta},2}, a_{\tilde{\boldsymbol{\beta}},1}, a_{\tilde{\boldsymbol{\beta}},2}, b_{\boldsymbol{\beta}}, b_{\tilde{\boldsymbol{\beta}}}, \mu)$. Therefore, the geometric program \eqref{Equ:ConstrainedGP} reads as follows:
{\small
\begin{eqnarray*}
& \inf\left\{\frac{1}{2}\mu + \frac{7}{10} \cdot b_{\boldsymbol{\beta}}^{\frac{10}{7}} \cdot \left(\frac{1}{5}\right)^{\frac{2}{7}} \cdot \left(\frac{1}{10}\right)^{\frac{1}{7}} \cdot (a_{\boldsymbol{\beta},1})^{-\frac{2}{7}} \cdot (a_{\boldsymbol{\beta},2})^{-\frac{1}{7}} + \frac{1}{5} \cdot b_{\tilde{\boldsymbol{\beta}}}^5 \cdot \left(\frac{1}{5}\right)^1 \cdot \left(\frac{3}{5}\right)^3 \cdot (a_{\tilde{\boldsymbol{\beta}},1})^{-1} \cdot (a_{\tilde{\boldsymbol{\beta}},2})^{-3} \right\}
\end{eqnarray*}}
such that the variables satisfy
\begin{eqnarray*}
a_{\boldsymbol{\beta},1} + a_{\tilde{\boldsymbol{\beta}},1} \leq 1 ,\; a_{\boldsymbol{\beta},2} + a_{\tilde{\boldsymbol{\beta}},2} \leq \mu \;\text{ and }\; 1 \leq b_{\boldsymbol{\beta}},\; \mu \leq b_{\tilde{\boldsymbol{\beta}}}\;. 
\end{eqnarray*}

We use the \textsc{Matlab} solver \textsc{CVX} to solve the program given above. The optimal solution is given by
\[
 (a_{\boldsymbol{\beta},1}, a_{\boldsymbol{\beta},2}, a_{\tilde{\boldsymbol{\beta}},1}, a_{\tilde{\boldsymbol{\beta}},2}, b_{\boldsymbol{\beta}}, b_{\tilde{\boldsymbol{\beta}}}, \mu) \ = \ (0.9105,0.0540,0.0895,0.0319,1.0000,0.0859,0.0859) \;.
\]
This leads to 
$$\gamma_{\rm sonc} \ \approx \ 0.5526$$
and hence $f_{\boldsymbol{\alp}(0)}-\gamma_{\rm sonc} \approx 0.4474$. Thus, we have 
$$f^{(8)}_{\rm sos} \ = \ f_{\boldsymbol{\alp}(0)}-\gamma_{\rm sonc} \ = \ s(f,g_1).$$
The equality $f_{\boldsymbol{\alp}(0)}-\gamma_{\rm sonc} =s(f,g_1)$ is not surprising, since the assumptions of Theorem \ref{Thm:ConstrainedGPequality} are satisfied. Thus, we get the optimal solution immediately via geometric programming whereas one needs $5$ relaxation steps via Lasserre relaxation. In this example both geometric programing and the Lasserre approach have a runtime below 1 second. 

Now, we demonstrate that our approach, in contrast to SDPs, is not sensitive to increasing the degree, since it is based on GPs. Namely, if we multiply all exponents  in $f$ and $g_1$ by 10, then the approaches differ significantly. By multiplying the exponents by 10 we have made a severe change to the problem since the term $x^{10}y^{10}$ is now a monomial square such that the exponent is a lattice point in the interior of the Newton polytope of the adjusted $G(\mu)$. Therefore, we have to ignore this term out when running the constrained optimization program \eqref{Equ:ConstrainedGP}. The adjusted program yields with \textsc{CVX} an output \texttt{NaN} in below one second. However, the reason is that it computes $\mu = 0$, which \textbf{is} the correct answer. Namely, after multiplying the exponents by 10, the only non monomial square terms are given by $g_1$. Thus, the optimal choice is $\mu = 0$, and we can see that the minimal value is attained at $(0,0)$ and $f_K^* = 1$ is given by the constant term of $f$.

In comparison, we have a runtime of approximately $1110$ seconds, i.e. approximately $18.5$ minutes with \textsc{Gloptipoly}. After this time \textsc{Gloptipoly} provides an output ``\texttt{Run into numerical problems.}''. It claims, however, to have solved the problem and provides the correct minimum $f_K^*=1$ at a minimizer $10^{-7} \cdot (-0.1057,0.1711)$, which, of course, is the origin up to a numerical error.

We remark that this particular artificial increase of the degree can, of course, be handled with SDP methods using pre-processing methods, e.g. a change of variables. The purpose of the rescaling of this and later examples is, however, to show that SONC/GP methods are \textit{inherently} not sensitive to an increase of the degree.
\label{Exa:Constrained2}
\endexa
\end{example}

\begin{figure}
\ifpictures
$\includegraphics[width=0.35\linewidth]{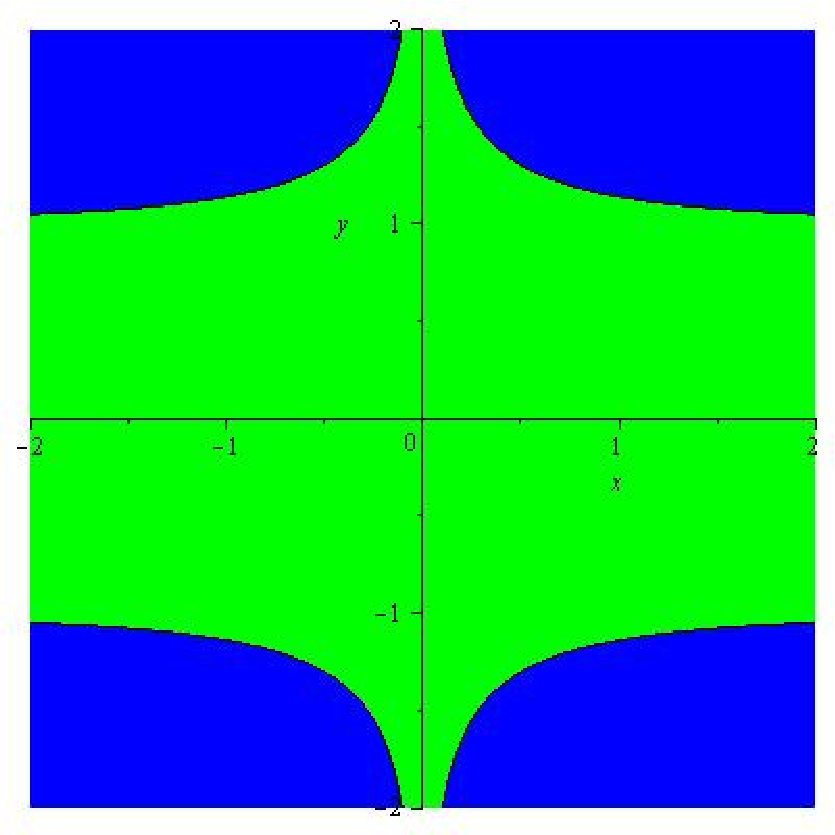}$
\fi
\caption{The feasible set for the constrained optimization problem in Example \ref{Exa:Constrained2} is the unbounded green (light) area.}
\label{Fig:FeasibleSet1}
\end{figure}

\begin{example}
Let $f = 1 + {x}^{2}{z}^{2} + {y}^{2}{z}^{2} + {x}^{2}{y}^{2} - 8xyz$ and $g_1 = {x}^{2}yz + x{y}^{2}z + {x}^{2}{y}^{2} - 2 + xyz$.
Using \textsc{Gloptipoly}, we get the following sequence of lower bounds: 
$$f_{\rm sos}^{(2)} = f_{\rm sos}^{(3)} = f_{\rm sos}^{(4)} = -\infty < f_{\rm sos}^{(5)} \approx -14.999.$$
However, one cannot certify the optimality via \textsc{Gloptipoly} in this case. Additionally, the sequence $f_{\rm sos}^{(d)}$ is not guaranteed to converge to $f_K^*$, since $K$ is unbounded. Symbolically, we were able to prove a global minimum of $f_K^* = -15$ with four global minimizers $(2,2,2), (-2,-2,2), (-2,2,-2), (2,-2,-2)$ using the quantifier elimination software \textsc{Synrac}, see \cite{Anai:Yanami:SYNRAC}. Now, we consider the approach via geometric programming instead of Lasserre relaxations. We have 
$$G(\mu) \ = \ (1+2\mu) + x^2z^2 + y^2z^2 + (1 - \mu)x^2y^2 + (- 8 -\mu )xyz  - \mu x^2yz - \mu xy^2z .$$
Therefore, $G(\mu)$ is an ST-polynomial for $\mu \in [0,1)$, and we have $\Delta(G)=\{\boldsymbol{\beta}, \overline{\boldsymbol{\beta}},\hat{\boldsymbol{\beta}}\}=$ $\{(1,1,1),(2,1,1),(1,2,1)\}$. Thus, our geometric program has the following 13 variables 
$$(a_{\boldsymbol{\beta},1},a_{\boldsymbol{\beta},2},a_{\boldsymbol{\beta},3},a_{\overline{\boldsymbol{\beta}},1},a_{\overline{\boldsymbol{\beta}},2},a_{\overline{\boldsymbol{\beta}},3},a_{\hat{\boldsymbol{\beta}},1},a_{\hat{\boldsymbol{\beta}},2},a_{\hat{\boldsymbol{\beta}},3},b_{\boldsymbol{\beta}},b_{\overline{\boldsymbol{\beta}}},b_{\hat{\boldsymbol{\beta}}}, \mu). $$
Hence, program \eqref{Equ:ConstrainedGP} is of the form
\begin{eqnarray*}
& & \inf\left\{0 \cdot \mu+ \frac{1}{4} \cdot b_{\boldsymbol{\beta}}^4 \cdot \left(\frac{1}{4}\right) \cdot \left(\frac{1}{4}\right) \cdot \left(\frac{1}{4}\right) \cdot (a_{\boldsymbol{\beta},1})^{-1} \cdot (a_{\boldsymbol{\beta},2})^{-1} \cdot (a_{\boldsymbol{\beta},3})^{-1}\right\}
\end{eqnarray*}
such that
$$
\begin{array}{cl}
 (1) & a_{\boldsymbol{\beta},1} + a_{\overline{\boldsymbol{\beta}},1} + a_{\hat{\boldsymbol{\beta}},1} \leq 1, \ a_{\boldsymbol{\beta},2} + a_{\overline{\boldsymbol{\beta}},2} + a_{\hat{\boldsymbol{\beta}},2} \leq 1, \ a_{\boldsymbol{\beta},3} + a_{\overline{\boldsymbol{\beta}},3} + a_{\hat{\boldsymbol{\beta}},3} + \mu \leq 1, \\
 (2) & \frac{1}{2} \cdot b_{\overline{\boldsymbol{\beta}}} \cdot \left(a_{\overline{\boldsymbol{\beta}},1}\right)^{-\frac{1}{2}} \cdot \left(a_{\overline{\boldsymbol{\beta}},3}\right)^{-\frac{1}{2}} \leq 1,\\
 & \frac{1}{2} \cdot b_{\hat{\boldsymbol{\beta}}} \cdot \left(a_{\hat{\boldsymbol{\beta}},2}\right)^{-\frac{1}{2}} \cdot \left(a_{\hat{\boldsymbol{\beta}},3}\right)^{-\frac{1}{2}} \leq 1, \\
 (3) & 8\cdot b_{\boldsymbol{\beta}}^{-1}\leq 1, \ \mu \cdot b_{\boldsymbol{\beta}}^{-1} \leq 1, \ \mu \cdot b_{\overline{\boldsymbol{\beta}}}^{-1} \leq 1, \ \mu \cdot b_{\hat{\boldsymbol{\beta}}}^{-1} \leq 1 \,.
\end{array}
$$
This leads to $\gamma_{\rm sonc}=\frac{1}{256}\cdot 8^4 =16$ and so $f_{\boldsymbol{\alp}(0)}-\gamma_{\rm sonc}=-15$. The runtime for this example is below 1 second. Multiplying the exponents of $f$ and $g_1$ by 10 yields the same results; the runtime for the geometric program remains \textbf{below 1 second}. In comparison, \textsc{Gloptipoly} yields
$$f_{\rm sos}^{(d)} \ = \ - \infty \text{ for } d \leq 19,$$
and provides a bound
$$f_{\rm sos}^{(20)} \ \approx \ -14.999$$
in the 20-th relaxation after $36563$ seconds, i.e. approximately \textbf{10.16 hours}. Moreover, although this bound is numerically equal to $f_K^*$, \textsc{Gloptipoly} was not able to certify that the correct bound was found.
\label{Exa:Constrained3}
\endexa
\end{example}

\begin{example}
Let $f = z^6 + x^4y^2 + x^2y^4 - 3x^2y^2z^2$ and $g_1 = x^2 + y^2 + z^2 - 1$. We obtain $G(\mu) = f - \mu g_1$. This problem is infeasible in the sense of program \eqref{Equ:ConstrainedGP}. Namely, condition \assum{} is never satisfied since for any $\mu > 0$ we have a vertex $(2,0,0)$ or $(0,2,0)$ of ${\rm New}(G(\mu))$ with a negative coefficient. Therefore, one can immediately conclude that $s(f,g_1)$ has to be obtained for $\mu = 0$. Thus, we have $s(f,g_1) = f_{\rm sonc}$. Since $f$ is the homogenized Motzkin polynomial we obtain immediately $f_{\rm sonc} = f_K^* = 0$. An analogous argumentation holds for the variation $\tilde{G}(\mu) = f + \mu g_1$.

It is well-known that SDP solvers have serious issues with optimizing $f$ for $g_1 \geq 0$ or $g_1 \leq 0$. For further information see \cite[Examples 5.3 and 5.4]{Nie:Jacobian}.
\endexa
\end{example}

In the last example in this section we show that for special simplices our geometric programming approach coincides with the one in \cite{ Ghasemi:Marshall:GP:Semialgebraic}.

\begin{example}
Suppose that $\New(G(\boldsymbol{\mu})) = \conv\{0, 2d\,\mathbf{e}_1,\dots,2d\,\mathbf{e}_n\}$. Hence, the Newton polytope is a $2d$-scaled standard simplex in $\mathbb R^n$, which is the case if the pure powers $x_j^{2d}$ for $1\leq j \leq n$ are present in the polynomial $f$ or in the constrained polynomials $g_i$. The corresponding polynomial $G(\boldsymbol{\mu})$ is an ST-polynomial; see Section \ref{SubSec:PrelimSONC}. Indeed, all examples in \cite[Example 4.8]{Ghasemi:Marshall:GP:Semialgebraic} are of that form and thus all of them are ST-polynomials. 

In this case the program \eqref{Equ:Constrained} coincides with the program (3) in \cite{Ghasemi:Marshall:GP:Semialgebraic}. One drawback of this setting is that the geometric programming bounds obtained from \eqref{Equ:Constrained} are at most as good as the bound $f_{\rm sos}^{(d)}$. Namely, if the Newton polytope of a circuit polynomial is a scaled standard simplex, then it is nonnegative if and only if it is a sum of squares; see \cite{Iliman:deWolff:Circuits} for further details. Thus, if we are in the setting of Ghasemi and Marshall and $G(\boldsymbol{\mu})$ is nonnegative, then it is a sum of squares of degree at most $2d$ which guarantees the existence of a decomposition in the sense of $f^{(d)}_{\rm sos}$; see \eqref{Equ:LasserreHierarchie}.

However, as we have shown in the previous examples, in the case of our program \eqref{Equ:Constrained} there exist also cases where the geometric programming bounds are better than $f_{\rm sos}^{(d)}$, since our approach is more general than in \cite{ Ghasemi:Marshall:GP:Semialgebraic}. The reason is that the cones of sums of nonnegative circuit polynomials and sums of squares do not contain each other (but both of them are contained in the cone of nonnegative polynomials); see \cite[Prop. 7.2]{Iliman:deWolff:Circuits}.

\endexa
\label{Exa:Constrained:GhasemiMarshallCase}
\end{example}

We point out that we make \textbf{no} assumption about the feasible set $K$. In particular, it is not assumed to be compact as it is in the classical setting via Lasserre relaxations in order to guarantee convergence of the relaxations. However, the crucial point in our setting so far is that $G(\boldsymbol{\mu})$ has to be an ST-polynomial. In the following Section \ref{Sec:NonSTPolynomials} we lay the foundation for the usage of our geometric programming approach also for non-ST-polynomials. 

But even if $G(\boldsymbol{\mu})$ is not an ST-polynomial, then we can enforce it to be an ST-polynomial in the case of a compact $K$. This can be achieved by adding a redundant constraint $g_{s+1} = x_1^{2d}+\ldots +x_{n}^{2d}+c$ for $c\in\mathbb R$ to the feasible set $K$. In  consequence $\New(G(\boldsymbol{\mu}))$ is a $2d$-scaled standard simplex and by the previous example our approach coincides with the one in \cite{Ghasemi:Marshall:GP:Semialgebraic}. Hence, the Lasserre relaxation cannot be outperformed in quality anymore. However, our approach can still have the better runtime. It would be interesting to add other redundant inequalities to $K$ such that the corresponding bounds are better than the ones obtained via Lasserre relaxations. Unfortunately, no systematic way is known so far.

Furthermore, we consider in this paper numerical methods to certify nonnegativity of polynomials. Often it is desirable to obtain exact solutions, i.e. a symbolic certification. It would be interesting to find symbolic certificates from our provided numerical solution. For SOS polynomials this has been studied e.g. by Peyrl and Parrilo \cite{Peyrl:Parrilo} and by Kaltofen, Li, Yang, and Zhi \cite{Kaltofen:Li:Yang:Zhi}.

\section{Optimization for Non-ST-Polynomials}
\label{Sec:NonSTPolynomials}

The goal of this section is to provide a first approach to tackle optimization problems (both constrained and unconstrained) which cannot be expressed as a single ST-polynomial using the methods developed in \cite{Iliman:deWolff:Circuits,Iliman:deWolff:GP} and in Section \ref{Sec:RelaxToGP} in this article. A more careful investigation of these general types of nonnegativity problems will be content of a follow-up article.\\

We start with the case of global nonnegativity for arbitrary polynomials via SONC certificates. We recall the following statement from \cite[Definition 7.1 and Proposition 7.2]{Iliman:deWolff:Circuits}, which immediately follows from Section \ref{SubSec:PrelimSONC}.
\begin{fact}
Let $f \in \R[\mathbf{x}]$ and assume that there exist SONC polynomials $g_1,\ldots,g_k$ and positive real numbers $\mu_1,\ldots,\mu_k$ such that $f = \sum_{i = 1}^k \mu_i g_i$. Then $f$ is nonnegative.
\end{fact}

Of course, if a SONC decomposition exists, then it is not obvious how to find it in general. For ST-polynomials we know that we can find a SONC decomposition via the geometric optimization problem described in Theorem \ref{Thm:GPUnconstrainedCase}. Thus, we investigate a general polynomial $f \in \R[\mathbf{x}]$ supported on a set $A \subset \N^n$ satisfying \assum. We denote
\begin{eqnarray*}
  f & = & \sum_{j = 0}^d f_{\boldsymbol{\alp}(j)} \mathbf{x}^{\boldsymbol{\alp}(j)} + \sum_{\boldsymbol{\beta} \in \Delta(f)} f_{\boldsymbol{\beta}} \mathbf{x}^{\boldsymbol{\beta}} 
\end{eqnarray*}
such that $f_{\boldsymbol{\alp}(j)} \mathbf{x}^{\boldsymbol{\alp}(j)}$ are monomial squares. By \assum{}, $V(A)$ are the vertices of $\New(f)$ and we have $V(A) \subseteq \{\boldsymbol{\alp}(0),\ldots,\boldsymbol{\alp}(d)\}$; equality, however, is not required here: $\{\boldsymbol{\alp}(0),\ldots,\boldsymbol{\alp}(d)\}$ can also contain exponents of monomial squares in $\Delta(A) \setminus \Delta(f)$ which are not vertices of $\conv(A)$. For simplicity we assume in what follows that the affine span of $A$ is $n$-dimensional. We proceed as follows:

\begin{enumerate}
 \item Choose a triangulation $\struc{T_1,\ldots,T_{k}}$ of exponents $\boldsymbol{\alp}(0),\ldots,\boldsymbol{\alp}(d) \in A$ corresponding to the monomial squares.
 \item Compute the induced covering $\struc{A_1,\ldots,A_{k}}$ of $A$ given by $A_i = A \cap T_i$ for $1 \leq i \leq k$.
 \item Assume that $\boldsymbol{\beta} \in \Delta(f) \subset A$ is contained in more than one of the $A_i$'s. Let without loss of generality $\boldsymbol{\beta} \in A_1,\ldots,A_l$ with $1 < l \leq k$. Then we choose $\struc{f_{\boldsymbol{\beta},1},\ldots,f_{\boldsymbol{\beta},l}} \in \R$ such that $\sum_{i = 1}^l f_{\boldsymbol{\beta},i} = f_{\boldsymbol{\beta}}$ and $\sign(f_{\boldsymbol{\beta},i}) = \sign(f_{\boldsymbol{\beta}})$ for all $1 \leq i \leq l$. We proceed analogously for $\boldsymbol{\alp}(0),\ldots,\boldsymbol{\alp}(d)$.
 \item Define new polynomials $g_1,\ldots,g_{k}$ such that
 \begin{eqnarray*}
    \struc{g_i} & = & \sum_{\boldsymbol{\beta} \in A_i} f_{\boldsymbol{\beta},i} \mathbf{x}^{\boldsymbol{\beta}}.
 \end{eqnarray*}
\end{enumerate}

Note that by (1) and (2) the covering $A_i$ is a set of integer tuples such that $\conv(A_i)$ is a simplex with even vertices and $A_i$ contains no even points corresponding to monomial squares except for the vertices of $\conv(A_i)$. Thus, by (2)--(4) we see that all $g_i$ are ST-polynomials, since the signs of the $f_{\boldsymbol{\beta},i}$ are identical with the signs of the coefficients of $f$. Therefore, monomial squares $f_{\boldsymbol{\alp}(j)} \mathbf{x}^{\boldsymbol{\alp}(j)}$ of $f$ get decomposed into a sum of monomial squares $\sum_{i = 1}^k f_{\boldsymbol{\alp}(j),i} \mathbf{x}^{\boldsymbol{\alp}(j)}$   such that each individual monomial square $f_{\boldsymbol{\alp}(j),i} \mathbf{x}^{\boldsymbol{\alp}(j)}$ is a term of exactly one $g_i$. We proceed analogously for the terms $f_{\boldsymbol{\beta}} \mathbf{x}^{\boldsymbol{\beta}}$. 
Additionally, it follows by construction that $f = \sum_{i = 1}^{k} g_i$. We apply the GP proposed in Corollary \ref{Cor:GP} on each of the $g_i$ with respect to a monomial square $f_{\boldsymbol{\alp}(j),i}\mathbf{x}^{\boldsymbol{\alp}(j)}$, which is a vertex of $\New(g_i) = \conv(A_i)$ (not necessarily the same $\boldsymbol{\alp}(j)$ for every $g_i$); we denote the minimizer by $\struc{m^*_i}$. We make the following observation about these minimizers which was similarly already pointed out in \cite[Section 3]{Iliman:deWolff:Circuits}:

\begin{lemma}
Let $f \in \R[\mathbf{x}]$ be a nonnegative circuit polynomial. Let $b_{\boldsymbol{\alp}} \mathbf{x}^{\boldsymbol{\alp}}$ be a monomial with $b_{\boldsymbol{\alp}} > 0$ and $\boldsymbol{\alp} \in (2\Z)^{n}$. Then $b_{\boldsymbol{\alp}} \mathbf{x}^{\boldsymbol{\alp}} \cdot f$ is also a nonnegative circuit polynomial.
\label{Lem:MinimizerLaurentPolynomial}
\end{lemma}

Note particularly that if $\mathbf{v} \in (\R^*)^n$ satisfies $f(\mathbf{v}) = 0$, then $(b_{\boldsymbol{\alp}} \mathbf{x}^{\boldsymbol{\alp}} \cdot f)(\mathbf{v}) = 0$.

\begin{proof}
It is easy to see that all conditions for \assum{} as well as the conditions (ST1) and (ST2) remain valid for $b_{\boldsymbol{\alp}} \mathbf{x}^{\boldsymbol{\alp}} \cdot f$. Thus, $b_{\boldsymbol{\alp}} \mathbf{x}^{\boldsymbol{\alp}} \cdot f$ still is a circuit polynomial and since $b_{\boldsymbol{\alp}} \mathbf{x}^{\boldsymbol{\alp}} \geq 0$ it is also nonnegative. 
\end{proof}

\begin{prop}
 Let $f$, $g_1,\ldots,g_{k}$, and $m^*_i$ be as explained above. Assume for $i = 1,\ldots,k$ that $m^*_i$ corresponds to the monomial square $f_{\boldsymbol{\alp}(j),i} \mathbf{x}^{\boldsymbol{\alp}(j_i)}$ with $\boldsymbol{\alp}(j_i) \in \{\boldsymbol{\alp}(1),\ldots,\boldsymbol{\alp}(d)\} \cap V(A_i)$. Then	
 $f - \sum_{i = 1}^{k} m^*_i \mathbf{x}^{\boldsymbol{\alp}(j_i)}$
is a SONC and hence nonnegative. Thus, the $m^*_i$ provide bounds for the coefficients $f_{\boldsymbol{\alp}(j),i}$ for $f$ to be nonnegative. Particularly, if for $i = 1,\ldots,l$ with $l \leq k$ the exponents $\boldsymbol{\alp}(j_i)$ are the origin, then  $f_{\boldsymbol{\alp}(0)} - \sum_{i = 1}^{l} m^*_i$ is a lower bound for $\struc{f^*} = \sup\{\gamma \in \R \ | \ f - \gamma \geq 0\}$.
\end{prop}

\begin{proof}
 By construction, we know that $g_i - m^*_i \mathbf{x}^{\boldsymbol{\alp}(j_i)}$ is a SONC. Thus, $f - \sum_{i = 1}^{k} m^*_i \mathbf{x}^{\boldsymbol{\alp}(j_i)} = \sum_{i = 1}^{k} g_i - m^*_i \mathbf{x}^{\boldsymbol{\alp}(j_i)}$ is a SONC, too. The last part follows by the definitions of the $m_i^*$'s and $f^*$.
\end{proof}

Note that the decomposition of $f$ into the $g_i$'s is not unique. First, the triangulation in (1) is not unique in general. And, second, the decomposition of the terms in (3) is arbitrary. Note also that there exist several monomial squares which appear in more than one $g_i$, since membership in $A_i$ is given by the chosen triangulation and every simplex $T_1$ intersects at least one other simplex $T_2$ in an $n-1$ dimensional face, which means that $A_1 \cap A_2$ contains at least $n$ even elements. As mentioned in the introduction, the problem to identify an optimal triangulation and an optimal decomposition of coefficients will be discussed in a follow-up article.\\

We provide some examples to show how this generalized approach can be used in practice.

\begin{example}
Let $f = 6 + x_1^2x_2^6 + 2x_1^4x_2^6 + 1x_1^8x_2^2 -1.2x_1^2x_2^3 -0.85x_1^3x_2^5 -0.9x_1^4x_2^3 -0.73x_1^5x_2^2 -1.14x_1^7x_2^2$. 
We choose a triangulation
\begin{eqnarray*}
 \{\alert{\mathbf{(0,0)}},\alert{\mathbf{(2,6)}},\alert{\mathbf{(4,6)}},(2,3),(3,5)\},\{\alert{\mathbf{(0,0)}},\alert{\mathbf{(4,6)}},\alert{\mathbf{(8,2)}},(2,3),(4,3),(5,2),(7,2)\}.
\end{eqnarray*}
Here and in what follows the vertices of each simplex are printed in red (bold). For the corresponding Newton polytope see Figure \ref{Fig:NewtonPolytopes1}. We split the coefficients equally among the two triangulations and obtain two ST-polynomials
\begin{eqnarray*}
	g_1 & = & 3 + x_1^2x_2^6 + x_1^4x_2^6 -0.6x_1^2x_2^3 -0.85x_1^3x_2^5, \text{ and } \\
	g_2 & = & 3 + x_1^4x_2^6 + 1x_1^8x_2^2 -0.6x_1^2x_2^3 -0.9x_1^4x_2^3 -0.73x_1^5x_2^2 -1.14x_1^7x_2^2.
\end{eqnarray*}
Using \textsc{CVX}, we apply the GP from Corollary \ref{Cor:GP} and obtain optimal values $m_1^* = 0.2121$, $m_2^* = 2.5193$, and a SONC decomposition
{\small
\begin{eqnarray*}
\begin{array}{lclc}
0.173 + \eps x_1^2x_2^6 + 0.522x_1^4x_2^6 -0.6x_1^2x_2^3 & + &
0.04 + x_1^2x_2^6 + 0.478x_1^4x_2^6 -0.85x_1^3x_2^5 & + \\
0.427 + 0.211x_1^4x_2^6 +\eps x_1^8x_2^2 -0.6x_1^2x_2^3 & + & 
0.663 + 0.436x_1^4x_2^6 + 0.085x_1^8x_2^2 -0.9x_1^4x_2^3 & + \\
0.753 + 0.186x_1^4x_2^6 + 0.177x_1^8x_2^2 -0.73x_1^5x_2^2 & + &
0.676 + 0.167x_1^4x_2^6 + 0.738x_1^8x_2^2 -1.14x_1^7x_2^2, & \\
\end{array}
\end{eqnarray*}
}
with $\eps < 10^{-10}$, i.e. $\eps$ is numerically zero. Namely, $(2,3)$ is located on the segment given by $\alert{\mathbf{(0,0)}}$ and $\alert{\mathbf{(4,6)}}$ and thus $\alert{\mathbf{(2,6)}}$ and $\alert{\mathbf{(8,2)}}$ have coefficients zero in the convex combinations of the point $(2,3)$. 

Thus, the optimal value $f_{\rm sonc}$, which provides us a lower bound for $f^*$, is $f_{\rm sonc} \approx 6 - 2.731 = 3.269$. In comparison, via Lasserre relaxation one obtains an only slightly better optimal value $f^*=3.8673$.

Our GP based bound can be improved significantly via making small changes in the distribution of the coefficients. For example, if one decides not to split the coefficient of the term $x_1^2x_2^3$ among $g_1$ and $g_2$ equally, but to put the entire weight of the coefficient into $g_1$, i.e.,
\begin{eqnarray*}
 	\tilde{g}_1 & = & 3 + x_1^2x_2^6 + x_1^4x_2^6 -1.2x_1^2x_2^3 -0.85x_1^3x_2^5, \text{ and } \\
	\tilde{g}_2 & = & 3 + x_1^4x_2^6 + 1x_1^8x_2^2 -0.9x_1^4x_2^3 -0.73x_1^5x_2^2 -1.14x_1^7x_2^2,
\end{eqnarray*}
then this yields to an improved bound $\tilde{f}_{\rm sonc} \approx 3.572$.
\label{Exa:Triangulation1}
\endexa
\end{example}

\begin{figure}
 \ifpictures
\includegraphics[width=0.45\linewidth]{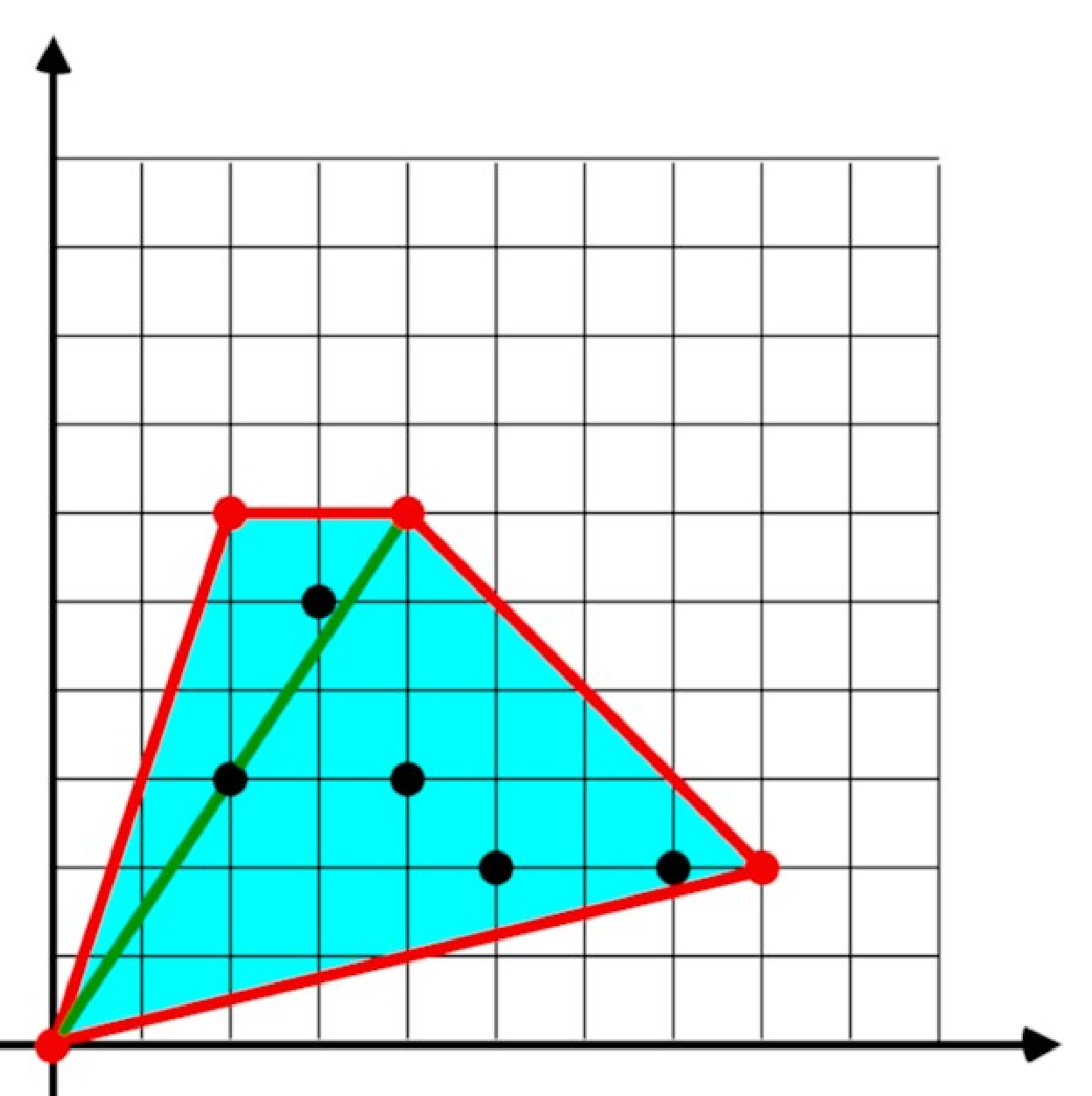} \quad
\includegraphics[width=0.45\linewidth]{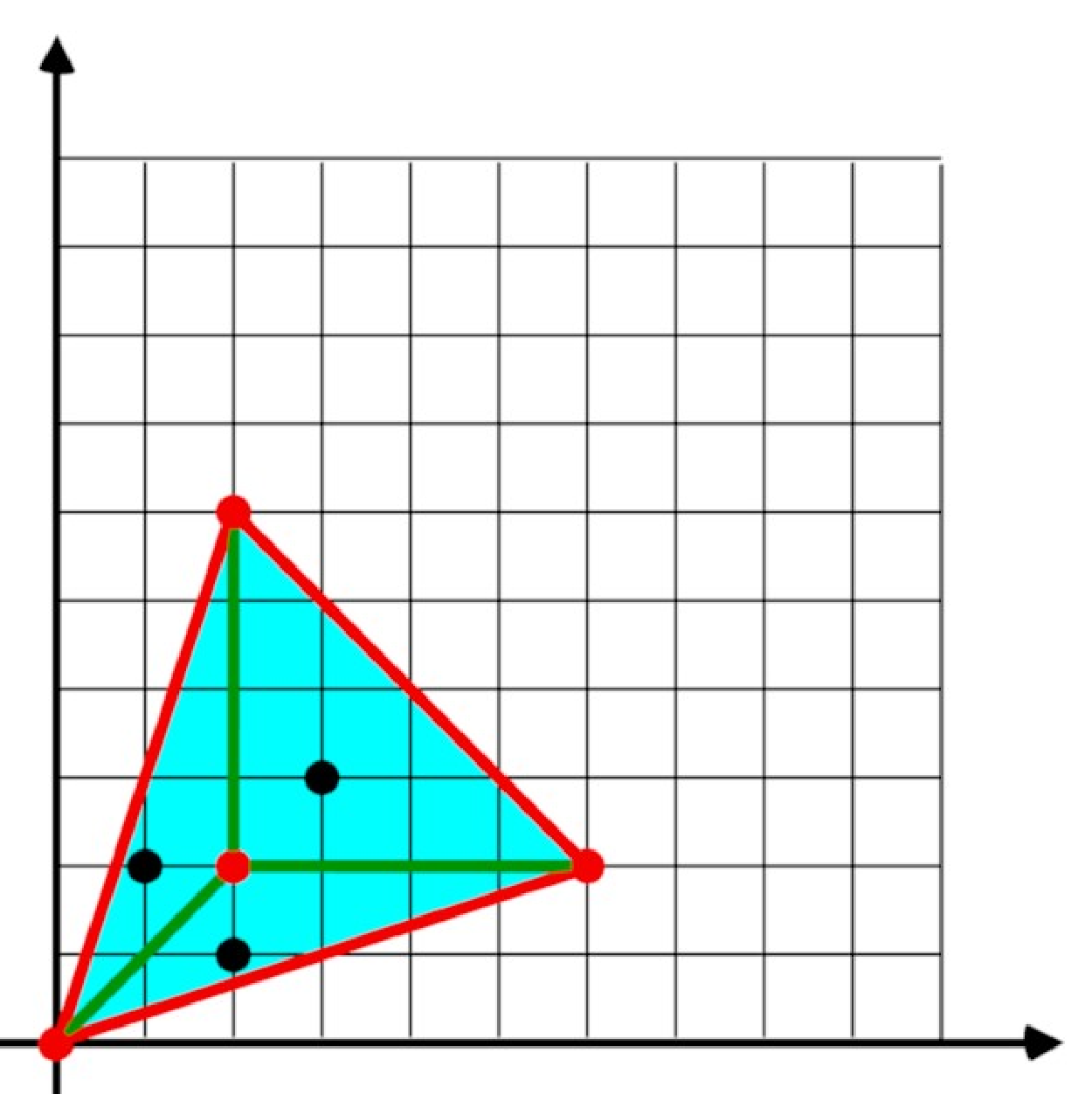} \quad
\fi
\caption{The Newton polytopes of the polynomials in the Examples \ref{Exa:Triangulation1} and \ref{Exa:Triangulation2} and their triangulations.}
\label{Fig:NewtonPolytopes1}
\end{figure}

The next example shows that we can use the approach of this section to take monomial squares into account, which are not vertices of the Newton polytope of the polynomial which we intend to minimize.

\begin{example}
 Let $f = 1 + 3x_1^2x_2^6 + 2x_1^6x_2^2 + 6x_1^2x_2^2 - x_1x_2^2 - 2x_1^2x_2 - 3x_1^3x_2^3$. We choose a triangulation
 \begin{eqnarray*}
  \{\alert{\mathbf{(0,0)}},\alert{\mathbf{(2,2)}},\alert{\mathbf{(2,6)}},(1,2)\},\{\alert{\mathbf{(0,0)}},\alert{\mathbf{(2,2)}},\alert{\mathbf{(6,2)}},(2,1)\},\{\alert{\mathbf{(2,2)}},\alert{\mathbf{(2,6)}},\alert{\mathbf{(6,2)}},(3,3)\}.
 \end{eqnarray*}
 For the corresponding Newton polytope see Figure \ref{Fig:NewtonPolytopes1}. First, we split the coefficients equally among the three triangulations such that we obtain
\begin{eqnarray*}
 g_1 & = & 0.5 + 1.5 x_1^2x_2^6 + 2x_1^2x_2^2 - x_1x_2^2, \\
 g_2 & = & 0.5 + 1 x_1^6x_2^2 + 2x_1^2x_2^2 - 2x_1^2x_2, \\
 g_3 & = &  1.5 x_1^2x_2^6 + 1x_1^6x_2^2 + 2x_1^2x_2^2 - 3x_1^3x_2^3.
\end{eqnarray*}
All three $g_i$ have a joint monomial $x_1^2x_2^2$. For all $i = 1,2,3$ we compute the maximal $b_i > 0$ such that $g_i - b_i x_1^2x_2^2$ is a nonnegative circuit polynomial. This yields a bound for the coefficient of $x_1^2x_2^2$ certifying that $f$ is a SONC and hence nonnegative. We could apply the GP from Corollary \ref{Cor:GP}, but since all $g_i$ are circuit polynomials we can compute the corresponding circuit numbers symbolically. We obtain with Theorem \ref{Thm:Positiv}:
{\footnotesize
\begin{eqnarray*}
 \Theta_{g_1}(1,2) & = & \left(\frac{1/2}{1/2}\right)^{\frac{1}{2}} \cdot \left(\frac{3/2}{1/4}\right)^{\frac{1}{4}} \cdot \left(\frac{2 - b_1}{1/4}\right)^{\frac{1}{4}} \ = \ \sqrt[4]{4 \cdot 4 \cdot 3/2 \cdot (2 - b_1)} \ = \ 2 \sqrt[4]{3/2 \cdot (2 - b_1)}, \\
\Theta_{g_2}(2,1) & = & \left(\frac{1/4}{1/2}\right)^{\frac{1}{2}} \cdot \left(\frac{1/2}{1/4}\right)^{\frac{1}{4}} \cdot \left(\frac{1 - 1/2 \cdot b_2}{1/4}\right)^{\frac{1}{4}} \ = \ \sqrt[4]{1/4 \cdot 2 \cdot 4(1 - 1/2 \cdot b_2)} \ = \ \sqrt[4]{2 - b_2}, \text{ and } \\
\Theta_{g_3}(3,3) & = & \left(\frac{1/2}{1/4}\right)^{\frac{1}{4}} \cdot \left(\frac{1/3}{1/4}\right)^{\frac{1}{4}} \cdot \left(\frac{1/3 (2 - b_3)}{1/2}\right)^{\frac{1}{2}} \ = \ \sqrt[4]{2 \cdot 4/3} \cdot \sqrt{2/3 \cdot (2 - b_3)} \ = \ 2 \sqrt[4]{2/27} \sqrt{2 - b_3}.
\end{eqnarray*}}
This provides solutions:
\begin{eqnarray*}
 2 \sqrt[4]{3/2 \cdot (2 - b_1)} \ \geq \ 1 & \Lera & 3/2 \cdot (2 - b_1) \ \geq \ 1/16 \ \Lera \ b_1 \ \leq \ 47/24, \\
 \sqrt[4]{2 - b_2} \ \geq \ 1 & \Lera & b_2 \ \leq \ 1, \\
 2 \sqrt[4]{2/27} \sqrt{2 - b_3} \ \geq \ 1 & \Lera & \sqrt{2/27} \cdot (2 - b_3) \ \geq \ 1/4 \ \Lera \ b_3 \ \leq \ 2 - \sqrt{27}/(2\sqrt{2}).
\end{eqnarray*}
Hence, we obtain the following bound for the coefficient of $x_1^2x_2^2$:
\begin{eqnarray*}
 6 - (47/24 + 1 + 2 - \sqrt{27}/(4\sqrt{2})) & \approx & 6 - 4.03977468 \ \approx \ 1.96.
\end{eqnarray*}
A double check with the \textsc{CVX} solver for GPs yields the same value in approximately $0.753$ seconds.

We want to compute a bound for $f^*$. We choose the same triangulation and the same split of coefficients as before, but now we optimize the constant term in $g_1$ and $g_2$, and we optimize the coefficient of $x_1^2x_2^6$ in $g_3$. After a runtime of approximately $0.6657$ seconds we obtain optimal values $0.0722, 0.3536$, and $0.3164$. Thus, we found a lower bound for the constant term given by
\begin{eqnarray*}
 m_1^* + m_2^* & \approx & 0.0722 + 0.3536 \ = \ 0.4268.
\end{eqnarray*}
The corresponding optimal SONC decomposition is given by
\begin{eqnarray*}
\begin{array}{lclc}
0.0722 + 1.5x_1^2x_2^6 + 2x_1^2x_2^2-x_1^1 x_2^2 & + & 0.3536 + 1x_1^6x_2^2 + 2x_1^2x_2^2-1x_1^2 x_2^1 & +  \\
0.3164x_1^2x_2^6 + 1x_1^6x_2^2 + 2x_1^2x_2^2-3x_1^3 x_2^3 & & & \\
\end{array}
\end{eqnarray*}
Thus, we obtain a bound for $f^*$ given by
\begin{eqnarray*}
 f_{\rm sonc} & = & 1 - 0.4268 \ = \ 0.5732
\end{eqnarray*}

We make a comparison and optimize $f$ with Lasserre relaxation. This yields an optimal value 
$$f_{\rm sos} \ = \ f^* \ \approx \ 0.8383.$$
Therefore, we want to improve our bound. We keep the triangulation, but we use another distribution of the coefficients among the polynomials $g_1,g_2$ and $g_3$ and define instead
\begin{eqnarray*}
 \tilde{g}_1 & = & 0.25 + 2x_1^2x_2^6 + 1.217x_1^2x_2^2 - 2x_1x_2^2 , \\
 \tilde{g}_2 & = & 0.75 + 1 x_1^6x_2^2 + 3.652x_1^2x_2^2 - 1x_1^2x_2, \\
 \tilde{g}_3 & = &  1 x_1^2x_2^6 + 1x_1^6x_2^2 + 1.13x_1^2x_2^2 - 3x_1^3x_2^3.
\end{eqnarray*}
Again, we optimize $\tilde{g}_1$ and $\tilde{g}_2$ with respect to the constant term and $\tilde{g}_3$ with respect to $x_1^2x_2^6$. We obtain optimal values $0.0801, 0.2616$, and $0.9912$. Thus, we are able to improve our bound for $f^*$ to
\begin{eqnarray*}
 \tilde{f}_{\rm sonc} & \approx & 1 - (0.0801 + 0.2616) \ = \ 0.6583.
\end{eqnarray*}
The corresponding optimal SONC decomposition is given by
\begin{eqnarray*}
\begin{array}{lclc}
0.0801 + 2x_1^2x_2^6 + 1.205x_1^2x_2^2-2x_1^1 x_2^2 & + & 0.2616 + 1x_1^6x_2^2 + 3.615x_1^2x_2^2-1x_1^2 x_2^1 & +  \\
0.991x_1^2x_2^6 + 1x_1^6x_2^2 + 2x_1^2x_2^2-3x_1^3 x_2^3. & & & \\
\end{array}
\end{eqnarray*}
\label{Exa:Triangulation2}
\endexa
\end{example}

\begin{figure}
 \ifpictures
\includegraphics[width=0.45\linewidth]{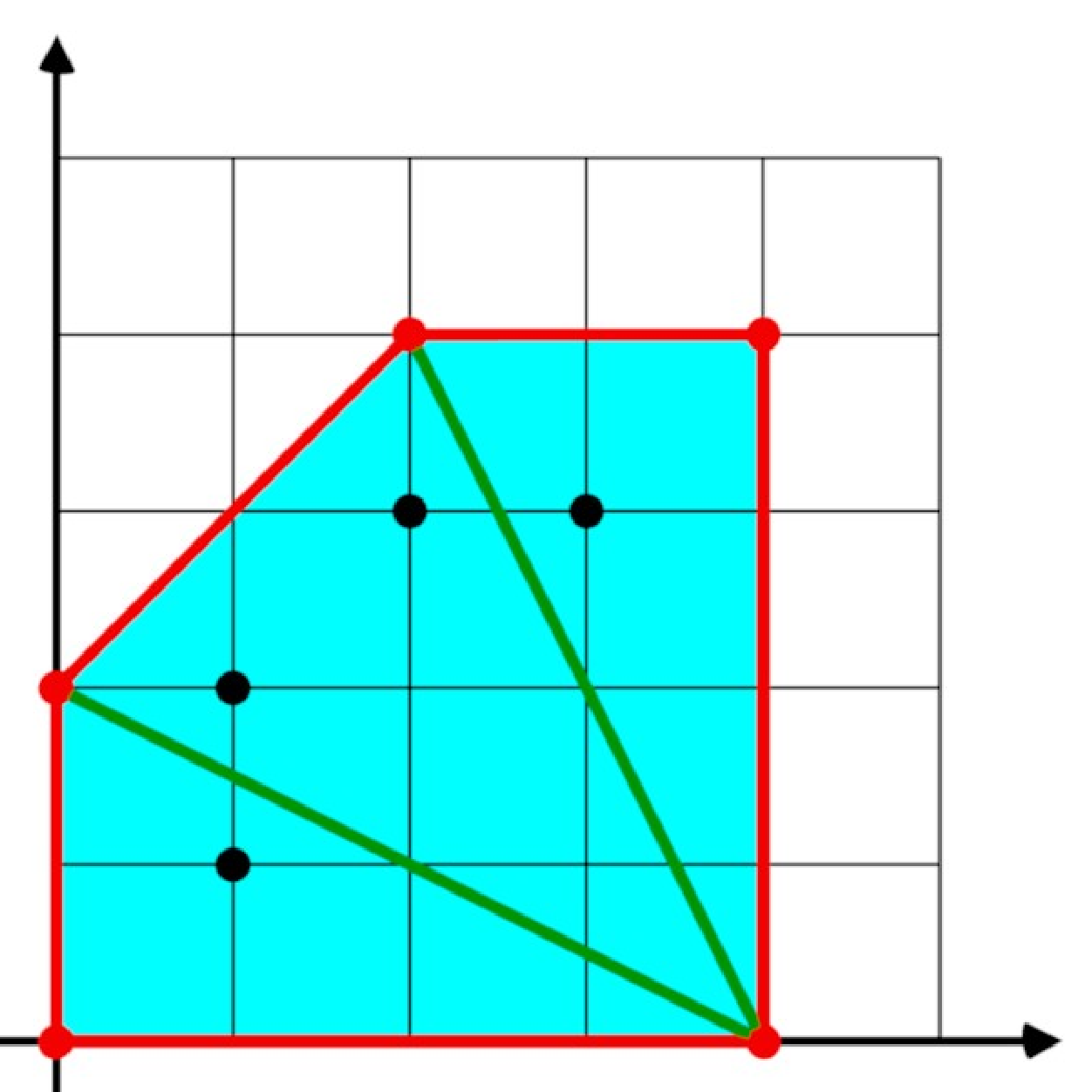} \quad
\includegraphics[width=0.45\linewidth]{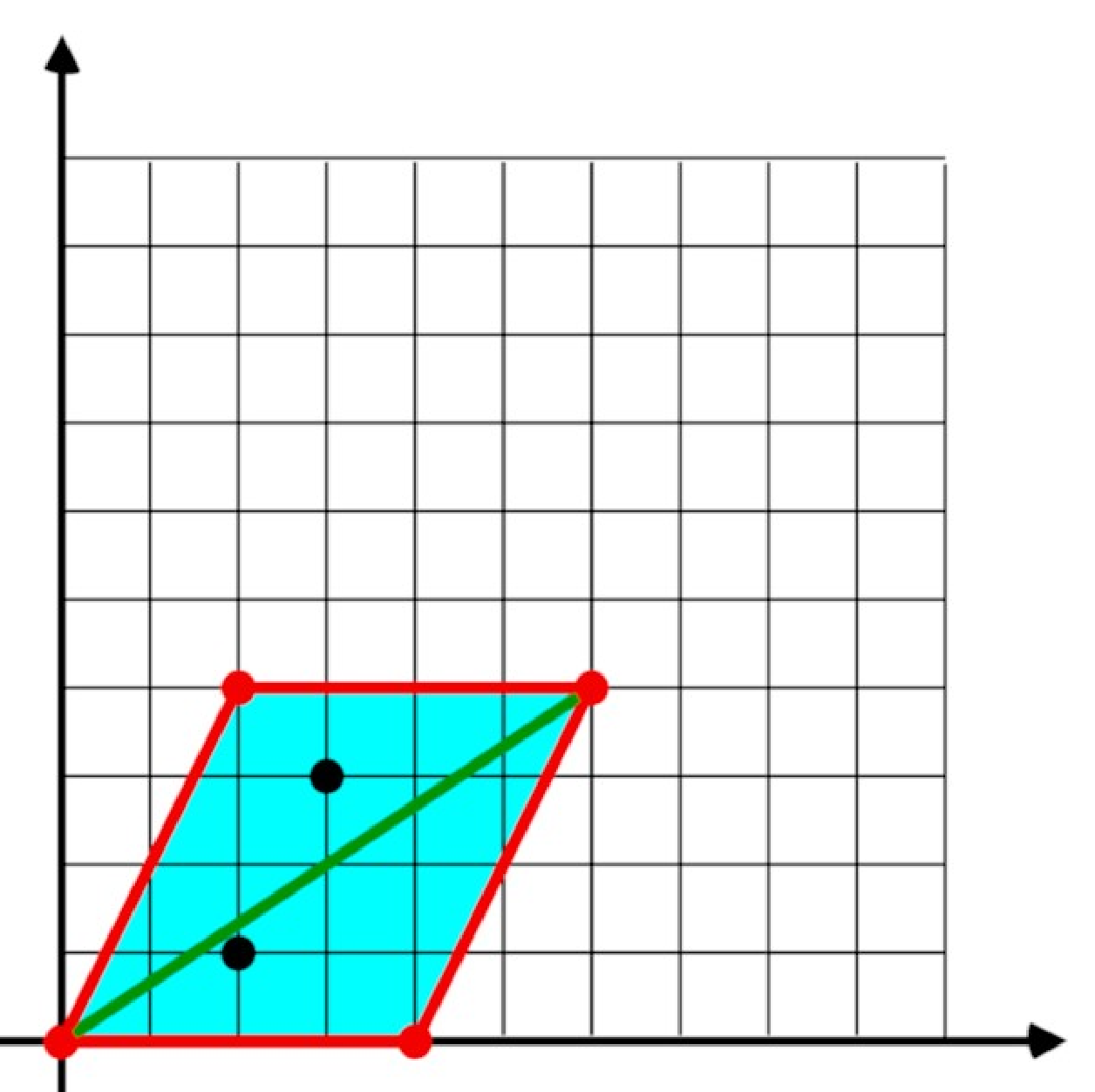} \quad
\fi
\caption{The Newton polytopes of the polynomials in the Examples \ref{Exa:Triangulation3} and \ref{Exa:Triangulation4} and their triangulations.}
\label{Fig:NewtonPolytopes2}
\end{figure}

We discuss a third example which shows that, in the case of global optimization, for the GP/SONC approach it is not necessary to optimize the constant term to obtain a bound for nonnegativity on the coefficients, but that in some cases it can be informative to focus on other vertices of the Newton polytope or on other monomial squares instead.

\begin{example}
Let $f = 1 + x_1^4 + x_2^2 + x_1^2x_2^4 + x_1^4x_2^4 - x_1x_2 - x_1x_2^2 - x_1^2x_2^3 - x_1^3x_2^3$. We choose a triangulation
\begin{eqnarray*}
 \{\alert{\mathbf{(0,0)}},\alert{\mathbf{(0,2)}},\alert{\mathbf{(4,0)}},(1,1)\},\{\alert{\mathbf{(0,2)}},\alert{\mathbf{(2,4)}},\alert{\mathbf{(4,0)}},(1,2),(2,3)\},\{\alert{\mathbf{(2,4)}},\alert{\mathbf{(4,0)}},\alert{\mathbf{(4,4)}},(3,3)\}.
\end{eqnarray*}
Again, we choose a decomposition of coefficients such that their values split equally. We obtain the following ST-polynomials
\begin{eqnarray*}
 g_1 & = & 1 + 1/3 \cdot x_1^4 + 1/2 \cdot x_2^2 - x_1x_2, \\
 g_2 & = & 1/3 \cdot x_1^4 + 1/2 \cdot x_1^2x_2^4 + 1/2 \cdot x_2^2 - x_1x_2^2 - x_1^2x_2^3, \\
 g_3 & = & 1/3 \cdot x_1^4 + 1/2 \cdot x_1^2x_2^4 + x_1^4x_2^4 - x_1^3x_2^3.
\end{eqnarray*}
$g_1$ and $g_3$ are circuit polynomials while $g_2$ contains two negative terms. For the corresponding Newton polytope see Figure \ref{Fig:NewtonPolytopes2}. Note that only the exponent $\alert{\mathbf{(4,0)}}$ is contained in the  support of all three ST-polynomials. Since $\alert{\mathbf{(4,0)}}$ is a monomial square which is a vertex of the convex hull of the three support sets, we optimize the corresponding coefficient in $g_1,g_2$ and $g_3$. Applying the GP from Corollary \ref{Cor:GP} yields optimal values
\begin{eqnarray*}
 m_1^* \ = \ 0.0625, \quad m_2^* \ = \ 4.2867, \quad \text{and} \quad  m_3^* \ = \ 0.0625.
\end{eqnarray*}
Since $m_2^* = 4.2867 > 1/3$ we found no certificate of nonnegativity for $f$. However, we find a SONC decomposition for $f$ if the coefficient $b_{(4,0)}$ of $x_1^4$ is at least $m_1^* + m_2^* + m_3^* = 4.412$. For this minimal choice of $b_{(4,0)}$ a SONC decomposition is given by
\begin{eqnarray*}
\begin{array}{lclc}
0.063x_1^4 + 1 + 0.5x_2^2 -1x_1^1 x_2^1 & + & 2.143x_1^4 + 0.4x_2^2 + 0.1x_1^2x_2^4 -1x_1^1 x_2^2 & + \\
2.143x_1^4 + 0.1x_2^2 + 0.4x_1^2x_2^4 -1x_1^2 x_2^3 & + & 0.063x_1^4 + 0.5x_1^2x_2^4 + 1x_1^4x_2^4 -1x_1^3 x_2^3 & \\
\end{array}
\end{eqnarray*}
\label{Exa:Triangulation3}
\endexa
\end{example}

Finally, we apply the new method to a constrained optimization problem using the methods developed in Section \ref{Sec:RelaxToGP}.

\begin{example}
Let $f=1+x^4+x^2y^4$ and $g=\frac{1}{2}+x^2y-x^6y^4-x^3y^3$. Hence, we obtain $G(\mu)=(1-\frac{1}{2}\mu)+x^4+x^2y^4+\mu x^6y^4 - \mu x^2y+\mu x^3y^3 $. Choosing the triangulation
\begin{eqnarray*}
  \{\alert{\mathbf{(0,0)}},\alert{\mathbf{(4,0)}},\alert{\mathbf{(6,4)}},(2,1)\},\{\alert{\mathbf{(0,0)}},\alert{\mathbf{(6,4)}},\alert{\mathbf{(2,4)}},(3,3)\},
 \end{eqnarray*}
we split the coefficients again, such that their values are equal. For the corresponding Newton polytope see Figure \ref{Fig:NewtonPolytopes2}. We obtain the ST-polynomials
\begin{eqnarray*}
 G_1(\mu) & = & \left(\frac{1}{2}-\frac{1}{4}\mu\right) + x^4 + \frac{1}{2}\mu x^6y^4 -\mu x^2y , \\
 G_2(\mu) & = & \left(\frac{1}{2}-\frac{1}{4}\mu\right) + x^2y^4 + \frac{1}{2}\mu x^6y^4 +\mu x^3y^3.
\end{eqnarray*}
Therefore, we see that the possible $\mu$ values to obtain ST-polynomials are $\mu \in [0,2)$. We optimize both polynomials with respect to the constant term and obtain $m^*_1=m^*_2=0$. The \textsc{CVX} solver yields \texttt{NaN} as an optimal value, since 0 is not positive. However, it solves the problem and computes values $0$ or $\eps < 10^{-200}$ for all variables, such that  $m^*_1=m^*_2=0$ follows. Hence, $f_{\boldsymbol{\alp}(0)}-m^*=1-0=1$ and because all of the assumptions in Theorem \ref{Thm:ConstrainedGPequality} are satisfied we know $s(f,g)=1$.

Checking this optimization problem with Lasserre relaxation, we get $f_{\rm sos}=f_K^*=1$, which approves the optimal value. Both, for the SDP and the GP we have runtimes below 1 second. 

Now, we tackle the same problem, but we multiply every exponent by 10, and we compare the runtimes again. For the GP we obtain the same result and the runtime remains \textbf{below 1 second}. For the SDP we obtain with \textsc{Gloptipoly} $f_{\rm sos}=f_K^*=1$ in approximately $\mathbf{5034.5}$ \textbf{seconds}, i.e. approximately $\mathbf{1.4}$ \textbf{hours}.

In a third approach we tackle the same problem, but we multiply the originally given exponents by 20. In this case \textsc{Gloptipoly} is not able to handle the given matrices anymore. In comparison, we still have a runtime  \textbf{below 1 second} for our GP providing the same bound as before.

Now, we also re-compute the example with the software \textsc{SOSTOOLS} \cite{sostools}, which, in contrast to \textsc{Gloptipoly}, can exploit sparsity patterns. For the original problem size, the runtime is 0.42 seconds, but already in the case that all exponents are multiplied by 10 we have a runtime of 302.7100 seconds, i.e. about \textbf{5 minutes}. 

Finally, we make a last comparison using the software \textsc{SparsePOP} \cite{SparsePOP}, which especially serves to compute sparse problems based on SOS. Here we start with a runtime of 0.4495 seconds for the original problem. For the case of all exponents of $f$ multiplied by $10$ we get a runtime of 81.5438 seconds, i.e. \textbf{1.359 minutes}. 
\label{Exa:Triangulation4}
\endexa
\end{example}

\bibliographystyle{amsalpha}
\bibliography{./gp}

\end{document}